\documentclass{article}
\usepackage{graphicx} 
\usepackage{amsthm}
\usepackage{amsfonts, amsmath, amssymb, amsthm}
\usepackage{cleveref}
\usepackage[utf8x]{inputenc}
\usepackage{url}
\usepackage{tikz-cd}

\def\PP{\mathbb{P}}

\newcommand{\Psf}{\mathsf{P}}

\def\U{\mathcal{U}}

\def\L{\mathcal{L}}
\def\R{\mathcal{R}}

\def\P{\mathcal{P}}
\def\G{\mathcal{G}}
\def\M{\mathcal{M}}
\def\Nc{\mathcal{N}}

\def\T{\mathcal{T}}
\def\S{\mathcal{S}}

\newcommand{\uh}{\upharpoonright}

\newcommand{\ISig}{\mathsf{I}\Sigma^0}

\newcommand{\IPi}{\mathsf{I}\Pi^0}
\newcommand{\IDelta}{\mathsf{I}\Delta^0}

\newcommand{\BSig}{\mathsf{B}\Sigma^0}
\newcommand{\BPi}{\mathsf{B}\Pi^0}

\newcommand{\RCA}[0]{\mathsf{RCA}}
\newcommand{\WKL}[0]{\mathsf{WKL}}

\newcommand{\ACA}[0]{\mathsf{ACA}}

\newcommand{\CAC}[0]{\mathsf{CAC}}
\newcommand{\ADS}[0]{\mathsf{ADS}}
\newcommand{\SADS}[0]{\mathsf{SADS}}
\newcommand{\EM}[0]{\mathsf{EM}}
\newcommand{\HEM}[0]{\mathsf{HEM}}

\newcommand{\RT}[0]{\mathsf{RT}}

\newcommand{\COH}[0]{\mathsf{COH}}
\newcommand{\BRT}[0]{\mathsf{BRT}}

\newcommand{\NN}[0]{\mathbb{N}}

\newcommand{\Definable}[2]{#1\mbox{-}\mathsf{Def}(#2)}

\newcommand{\card}{\operatorname{card}}

\def\qt#1{``#1''}%

\title{The reverse mathematics of\\ bounded Ramsey's theorem for pairs}
\date{\today}

\newtheorem*{statement}{Statement}
\newtheorem{theorem}{Theorem}
\numberwithin{theorem}{section}
\newtheorem{maintheorem}[theorem]{Main Theorem}
\newtheorem{lemma}[theorem]{Lemma}
\newtheorem{question}[theorem]{Question}
\newtheorem{proposition}[theorem]{Proposition}

\newtheorem{definition}[theorem]{Definition}
\newtheorem{corollary}[theorem]{Corollary}

\makeatletter
\newtheorem*{rep@theorem}{\rep@title}
\newcommand{\newreptheorem}[2]{%
\newenvironment{rep#1}[1]{%
 \def\rep@title{#2 \ref{##1}}%
 \begin{rep@theorem}}%
 {\end{rep@theorem}}}
\makeatother

\newreptheorem{theorem}{Theorem}
\newreptheorem{maintheorem}{Main Theorem}

\usepackage{xcolor}

\author{Quentin Le Houérou \and Ludovic Patey}

\begin{document}

\maketitle

\begin{abstract}
In this article, we study a degenerate version of Ramsey's theorem for pairs and two colors ($\RT^2_2$), in which the homogeneous sets for color~1 are of bounded size. By $\RT^2_2$, it follows that every such coloring admits an infinite homogeneous set for color~0. This statement, called $\BRT^2_2$, is known to be computably true, that is, every computable instance admits a computable solution, but the known proofs use $\Sigma^0_2$-induction ($\ISig_2$). We prove that $\BRT^2_2$ follows from the Erd\H{o}s-Moser theorem but not from the Ascending Descending sequence principle, and that its computably true version is equivalent to~$\ISig_2$ over~$\RCA_0$.
\end{abstract}

\section{Introduction}

Ramsey's theorem for pairs and two colors ($\RT^2_2$) is a combinatorial theorem which states that every coloring $f : [\NN]^2 \to 2$ admits an infinite $f$-homogeneous set. Here, given a set~$X$, $[X]^2$ denotes the set of unordered pairs over~$X$ and a set $H \subseteq X$ is \emph{$f$-homogeneous} if all the pairs in $[H]^2$ have the same color by~$f$. The statement $\RT^2_2$ was extensively studied from a meta-mathematical viewpoint, both in reverse mathematics and computability theory. In particular, Specker~\cite{specker_ramseys_1971} proved that $\RT^2_2$ is not \emph{computably true}, in that there is a computable coloring $f : [\NN]^2 \to 2$ with no infinite computable $f$-homogeneous set.

Many applications of Ramsey's theorem in mathematics yield degenerate colorings. In some cases, the specific features of the coloring can be exploited to yield simpler solutions, in a computable or proof-theoretic sense. For instance, a coloring $f : [\NN]^2 \to 2$ is \emph{transitive} for some color~$i < 2$ if for every~$x < y < z$ such that $f(x, y) = f(y, z) = i$, then $f(x, z) = i$. The restriction of $\RT^2_2$ to colorings which are transitive for one or both colors yield the Chain AntiChain ($\CAC$) and the Ascending Descending Sequence ($\ADS$) principles, respectively. The statements $\RT^2_2$, $\CAC$, and $\ADS$ form a strictly decreasing sequence in reverse mathematics~(see \cite[Corollary 3.12]{hirschfeldt_combinatorial_2007} and \cite[Theorem 1.7]{manuel_lerman_separating_2013}). 

In this paper, we study another common degenerate version of Ramsey's theorem for pairs and two colors, in which the homogeneous sets for color~1 have bounded size. By $\RT^2_2$, it follows that there exists an infinite homogeneous set for color~0. This is part of a larger project started with and driven by Frittaion, who introduced bounded Ramsey's theorem for colorings of arbitrary tuples. This project was pursued by Soldà~\cite[Section 4.1]{solda2021calibrating}, and then independently by Belanger, Shafer, and Yokoyama (unpublished).

\begin{statement}[Bounded $\RT^2_2$]
$\BRT^2_2$ is the statement \qt{For every coloring $f : [\NN]^2 \to 2$ and every~$\ell \in \NN$ such that there is no $f$-homogeneous set for color~1 of size~$\ell$, there is an infinite $f$-homogeneous set for color~0.}
\end{statement}

The statement $\BRT^2_2$ is significantly weaker than~$\RT^2_2$. From a computational viewpoint, contrary to Ramsey's theorem for pairs, $\BRT^2_2$ is computably true. However, as we shall see, the existence of a computable solution requires $\Sigma^0_2$-induction. Since $\RT^2_2$ does not imply the $\Sigma^0_2$-induction scheme~\cite{chong2017inductive} in reverse mathematics, the pure existence of a solution to an instance of $\BRT^2_2$ is strictly weaker than the existence of a computable solution. We shall study this bounded version of Ramsey's theorem for pairs using the framework of reverse mathematics.

\subsection{Reverse mathematics}

\emph{Reverse mathematics} is a foundational program whose goal is to find optimal axioms to prove theorems from ordinary mathematics. It uses the framework of subsystems of second-order arithmetic, with a base theory, $\RCA_0$, capturing \qt{computable mathematics}. The theory $\RCA_0$ contains the theory $\Psf^-$ of discretely ordered commutative semirings, together with the $\Delta^0_1$-comprehension scheme and the $\Sigma^0_1$-induction scheme. We now give a more precise account of those schemes.

\begin{definition}
Given a family of formulas $\Gamma$, the \emph{$\Gamma$-induction scheme} is defined for every $\Gamma$-formula $\varphi$ as
$$
\varphi(0) \wedge \forall x (\varphi(x) \to \varphi(x+1)) \to \forall y \varphi(y)
$$
\end{definition}
\noindent
We write $\mathsf{I}\Gamma$ for the theory $\Psf^-$ together with the $\Gamma$-induction scheme. The arithmetic hierarchy induces a hierarchy of induction schemes. By Paris and Kirby~\cite{paris_n-collection_1978}, $\IPi_n$ and $\ISig_n$ are equivalent for every $n \geq 0$, while $\ISig_{n+1}$ is strictly stronger than~$\ISig_n$.
\bigskip

The induction scheme for arithmetic formulas is equivalent modulo $\Psf^- + \ISig_0$ to the following collection (or bounding) scheme for the same family of formulas:

\begin{definition}
Given a family of formulas $\Gamma$, the \emph{$\Gamma$-collection} (or $\Gamma$-bounding) scheme is defined for every $\Gamma$-formula $\varphi$ as
$$
\forall a [(\forall x < a) \exists y \varphi(x, y) \rightarrow \exists b (\forall x < a)(\exists y < b) \varphi(x, y)]
$$
\end{definition}
\noindent
We write $\mathsf{B}\Gamma$ for the theory $\Psf^-$ together with the $\Sigma^0_0$-induction scheme and the $\Gamma$-collection scheme. By Paris and Kirby~\cite{paris_n-collection_1978}, $\BSig_{n+1}$ and $\BPi_n$ are equivalent for every $n \geq 1$, and the following hierarchy is strict: 
$$
\ISig_1 < \BSig_2 < \ISig_2 < \BSig_3 < \cdots
$$
\bigskip

Being a $\Delta^0_n$-predicate is a semantic notion, which depends either of a theory, or a structure. Because of this, one usually defines the $\Delta^0_n$-comprehension scheme using the following syntactical artifact.

\begin{definition}
The \emph{$\Delta^0_n$-comprehension scheme} is defined for every $\Sigma^0_n$-formula $\varphi$ and every $\Pi^0_n$-formula $\psi$ as 
$$
\forall x(\varphi(x) \leftrightarrow \psi(x)) \to \exists X \forall y (y \in X \leftrightarrow \varphi(y))
$$
\end{definition}

The theories $\RCA_0 + \ISig_n$ and $\RCA_0 + \BSig_n$ being $\Pi^1_1$-conservative extensions of $\ISig_n$ and $\BSig_n$, respectively, the previous hierarchy remains strict over $\RCA_0$.

Since the beginning of reverse mathematics, many theorems have been studied. It appeared that most theorems are either provable over $\RCA_0$, or provably equivalent over $\RCA_0$ to one of four systems of axioms, linearly ordered by logical implication~\cite{simpson_subsystems_2009}. However, some counter-example exist, the most important one being Ramsey's theorem for pairs and two colors~\cite{seetapun1995strength,liu2012rt22}. Since then, the study of Ramsey's theorem, and more generally of combinatorial theorems in reverse mathematics, became a very active research field.

\subsection{Main contributions}

As mentioned, this article focuses on the study of a degenerate version of $\RT^2_2$, in which the homogeneous sets for color~1 have bounded size.
By an old argument of Hirst~\cite[Theorem 6.8]{hirst1987combinatorics}, $\RCA_0 \vdash \BRT^2_2 \to \BSig_2$. Belanger and Yokoyama (unpublished) gave independently two proofs of $\RCA_0 \vdash \ISig_2 \to \BRT^2_2$ (see \Cref{prop:isig2-urt22}). On the other hand, $\BRT^2_2$ follows from $\RT^2_2$, which is known not to imply $\ISig_2$ over~$\RCA_0$ by Chong, Slaman, and Yang~\cite{chong2017inductive}. It follows that $\RCA_0 + \BRT^2_2 \not\vdash \ISig_2$.
Our first main contribution is the following theorem:

\begin{maintheorem}[Short version]\label[maintheorem]{main:bounded-non-comp-non-isigma2}
Every model of $\RCA_0 + \BSig_2 + \neg \ISig_2$ contains an instance~$f$ of $\BRT^2_2$ with no $\Delta_1(f)$-definable solution. 
\end{maintheorem}

It follows that $\RCA_0 + \BSig_2 \not \vdash \BRT^2_2$, hence the strength of $\BRT^2_2$ lies strictly in between $\ISig_2$ and $\BSig_2$. Furthermore, the statement \qt{$\BRT^2_2$ is computably true} is equivalent to $\ISig_2$ over~$\RCA_0$. Here, given a $\Pi^1_2$-problem~$\Psf$, the statement \qt{$\Psf$ is computably true} means that every $\Psf$-instance~$X$ admits a $\Delta_1(X)$-definable solution. This situation already appears in the literature, with the tree theorem for singletons~\cite{corduan_reverse_2010,chong_strength_2020} and bounded-width weak K\"onig's lemma~\cite{simpson2021fragments}. Over $\RCA_0 + \BSig_2$, both statements are strictly in between $\ISig_2$ and $\BSig_2$, and their computably true version is equivalent to $\ISig_2$.
\bigskip

By Bovykin and Weiermann~\cite{bovykin2017strength}, there exists a natural decomposition of Ramsey's theorem for pairs into the Erd\H{o}s-Moser theorem ($\EM$) and the Ascending Descending Sequence principle ($\ADS$). Over $\RCA_0$, the Erd\H{o}s-Moser theorem can be seen as the statement that every coloring $f : [\NN]^2 \to 2$ admits an infinite sub-coloring which is transitive for both colors, while $\ADS$ states that every such coloring admits an infinite $f$-homogeneous set. From a reverse mathematical viewpoint, these statements almost play the role of a split pair, in that the strongest known common consequences are $\BSig_2$ and the existence of hyperimmune functions (\cite{kreuzer2012primitive,hirschfeldt_combinatorial_2007,manuel_lerman_separating_2013}). Soldà~\cite[Theorem 4.1.10]{solda2021calibrating} proved that $\BRT^2_2$ is a consequence of $\EM$, and we prove that $\BRT^2_2$ does not follow from $\ADS$ over $\RCA_0$. More precisely, we prove the following theorem:

\begin{maintheorem}\label[maintheorem]{main:bsig2-not-isig2-ads-not-urt22}
Every countable topped model of $\RCA_0 + \BSig_2 + \neg \ISig_2$ can be $\omega$-extended into a model of $\RCA_0 + \ADS + \neg \BRT^2_2$.    
\end{maintheorem}

Here, a model is \emph{topped} if its second-order part is of the form $\{ Z : Z \leq_T Y \}$ for some fixed set~$Y$, and an \emph{$\omega$-extension} of model~$\M$ is a model~$\Nc$ obtained from~$\M$ by adding sets, but keeping the integers part unchanged. This proves that $\RCA_0 + \ADS \not \vdash \BRT^2_2$ in a strong sense: $\RCA_0 + \ADS + \neg \BRT^2_2$ is $\Pi^1_1$-conservative over $\RCA_0 + \BSig_2 + \neg \ISig_2$. This separation proof involves a formalized preservation of a combinatorial property by $\ADS$, using an effective construction in a weak model of arithmetic.



\section{$\BRT^2_2$ and partition theorems on partial orders}

This section is essentially a survey in which we give the best known proof that $\BRT^2_2$ is computably true, and relate bounded Ramsey's theorem for pairs to two well-known theorems from graph theory, namely, Dilworth's theorem and Mirsky's theorem. Both theorems and variants have already been studied in reverse mathematics.

As mentioned in the introduction, Hirst~\cite[Theorem 6.8]{hirst1987combinatorics} proved that $\BRT^2_2$ implies $\BSig_2$ over~$\RCA_0$, thanks to the equivalence between $\BSig_2$ and the infinite pigeonhole principle ($\RT^1$). This principle states, for every $k \in \NN$ and every coloring $g : \NN \to k$, the existence of an infinite $g$-homogeneous set, that is, an infinite set~$H \subseteq \NN$ on which~$g$ is constant. Given such a coloring $g : \NN \to k$, let $f(x, y) = 0$ if $g(x) = g(y)$ and 1 otherwise. There are no $f$-homogeneous set for color~1 of size~$k+1$, and any infinite $f$-homogeneous set for color~0 is $g$-homogeneous.

Versions of Ramsey's theorem for pairs in which all the infinite homogeneous sets have the same color were mentioned in reverse mathematics by Jockusch~\cite{jockusch_ramseys_1972} under the name \qt{unbalanced}, but bounded Ramsey's theorem was first introduced by Frittaion (unpublished), who proved $\BRT^2_2$ over $\RCA_0 + \ISig_3$. The upper bound was improved by Belanger and Yokoyama (unpublished) independently with two different proofs. The one by Yokoyama can be found in Soldà~\cite[Lemma 4.1.9]{solda2021calibrating}. We now give the proof of Belanger.

\begin{proposition}[Belanger and Yokoyama]\label[proposition]{prop:isig2-urt22}
$\RCA_0 \vdash \ISig_2 \to \BRT^2_2$.
\end{proposition}
\begin{proof}[Proof of Belanger]
Let $T \subseteq \NN^{<\NN}$ be the computable tree such that
\begin{itemize}
    \item[(1)] every $x \in \NN$ appears in some node~$\sigma \in T$;
    \item[(2)] every $\sigma \in T$ is $f$-homogeneous for color~0;
    \item[(3)] for every $\sigma \in T$ and every~$x, y \in \NN$ such that $\sigma \cdot x, \sigma \cdot y \in T$, $f(x, y) = 1$.
\end{itemize}
By (2) and by assumption, $T$ has finite depth, so by (1) and $\BSig_2$, there is some level~$\ell_0 \in \NN$ such that $\{ \sigma \in T : |\sigma| = \ell_0 \}$ is infinite. By $\ISig_2$, there is a least such level $\ell \leq \ell_0$. Since $T$ has only one node at level~0, namely $\langle \rangle$, $\ell > 0$. By minimality of~$\ell$, $\{ \sigma \in T : |\sigma| = \ell-1 \}$ is finite, so by $\BSig_2$, there is some~$\sigma \in T$ of length~$\ell-1$ with infinitely many immediate children. The set $\{ x \in \NN : \sigma \cdot x \in T \}$ is therefore infinite, and by (3), it is $f$-homogeneous for color~1.
\end{proof}

Note that if the bound to the 1-homogeneous sets is a standard integer, then so is the depth of the tree in \Cref{prop:isig2-urt22}, and the statement is provable over~$\RCA_0 + \BSig_2$ (see Soldà~\cite[Corollary 4.1.5]{solda2021calibrating}).
\smallskip

Bounded Ramsey's theorem for pairs and two colors can be formulated in terms of graphs, stating that every infinite graph with no sub-clique of some size~$k$ admits an infinite induced sub-anticlique. This statement has applications to partial orders, through their comparability and incomparability graphs.

\subsection{Dilworth's theorem}

Dilworth's theorem is a famous combinatorial theorem about partial orders, closely related to bounded Ramsey's theorem. The \emph{width} $w(\P)$ of a partial order $\P = (\NN, <_\P)$ is the size of its largest antichain. 

\begin{statement}[Dilworth's theorem]
\qt{For every partial order $\P = (\NN, <_\P)$ of finite width, there is a coloring $g : \NN \to w(\P)$ such that for every~$i < w(\P)$, $g^{-1}(i)$ is a $\P$-chain.}
\end{statement}

Such a coloring $g$ is called a \emph{chain partition} of~$\P$. Dilworth's theorem is classically proven for finite partial orders, but the infinitary version directly follows from it by a compactness argument as follows: Given a  partial order $\P = (\NN, <_\P)$ of finite width~$w$, one can define its incomparability graph $\G = (\NN, E)$, that is, the graph such that $\{x, y\} \in E$ iff $x$ and $y$ are $\P$-incomparable. Finite Dilworth's theorem essentially states that the graph $\G$ is locally $w$-colorable. By Gasarch and Hirst~\cite[Theorem 4]{gasarch1998reverse}, $\RCA_0 + \WKL$ proves that every locally $w$-colorable graph is $w$-colorable, and any such $w$-coloring is a chain partition of~$G$.
Thus, Dilworth's theorem follows from~$\RCA_0 + \WKL$.
Hirst~\cite[Theorem 3.23]{hirst1987combinatorics} studied Dilworth's theorem in reverse mathematics and proved that it is equivalent to~$\WKL$ over~$\RCA_0$, even restricted to partial orders of width~2.

\begin{proposition}[{Hirst~\cite[Theorem 3.23]{hirst1987combinatorics}}]
Over $\RCA_0$, Dilworth's theorem is equivalent to~$\WKL$.
\end{proposition}

Seeing a partial graph of finite width as a degenerate instance of~$\BRT^2_2$, we are interested in the existence of an infinite $\P$-chain. Given a chain partition $g : \NN \to k$, a final application of $\RT^1$ suffices to obtain such a chain. Therefore, $\RCA_0 + \WKL + \BSig_2$ is sufficient to prove that every partial order of finite width admits an infinite chain. Dilworth's theorem states that the chain partition of $\P$ equals its width, but for our purpose, it is sufficient to consider the following weaker statement.

\begin{statement}[Weak Dilworth's theorem]
\qt{For every partial order $\P = (\NN, <_\P)$ of finite width, there is some~$k \in \NN$ and a coloring $g : \NN \to k$ such that for every~$i < k$, $g^{-1}(i)$ is a $\P$-chain.}
\end{statement}

By Schmerl~\cite{schmerl2000graph}, if $2 \leq k \leq m$, then the statement \qt{Every locally $k$-colorable graph is $m$-colorable} is equivalent to $\WKL$ over~$\RCA_0$. Based on the proof of Dilworth's theorem, one would naturally expect that weak Dilworth's theorem remains equivalent to $\WKL$ over~$\RCA_0$. Surprisingly, Kierstead~\cite{kierstead1981effective} proved that every computable graph with width~$w$ admits a computable chain partition of size $(5^w-1)/4$.

\begin{theorem}[Kierstead~\cite{kierstead1981effective}]
Weak Dilworth's theorem is computably true.
\end{theorem}

The original proof by Kierstead~\cite{kierstead1981effective} involved strong induction axioms, but in their study of a theorem from Rival and Sands~\cite{rival1980adjacency}, Fiori-Carones, Marcone, Shafer, Soldà, with the help of Yokoyama, adapted the proof of Kierstead to formalize it over~$\RCA_0$ (see~\cite[Theorem 3.5]{fiori2024extraordinary}).

\begin{proposition}[{Fiori-Carones et al.~\cite{fiori2024extraordinary}}]
$\RCA_0$ proves weak Dilworth's theorem.
\end{proposition}

Kierstead's theorem can be stated in the modern formalism as the existence of an on-line algorithm for the chain partition problem, where an \emph{on-line algorithm} is one that produces its output progressively, as it receives its input, without having access to the entire input in advance. The bound $(5^w-1)/4$ obtained by Kierstead was recently improved by Bosek and Krawczyk~\cite{bosek2010subexponential,bosek2015subexponential} to $w^{14\log w}$ and by Bosek et al~\cite{bosek2018subexponential} to $w^{6.5\log w+7}$. See~\cite{bosek2012online} for a survey on the subject.
\smallskip

The proof of Kierstead's theorem is complex, but in the case where the partial order is up-growing, there exist significantly simpler algorithms (see Felsner~\cite{felsner1997online}). A presentation of a partial order is \emph{up-growing} if each new element is maximal among the elements inputted so far. The following proposition is essentially a reformulation of Felsner's algorithm~\cite{felsner1997online} in our formalism, in which an up-growing partial order can be seen as a coloring $f : [\NN]^2 \to 2$ which is transitive for color~0. Indeed, any such coloring can be seen as the partial order $\P = (\NN, <_\P)$ defined by $x <_\P y$ iff $x <_\NN y$ and $f(x, y) = 0$. 

\begin{proposition}[$\RCA_0$]\label[proposition]{prop:semi-transitive-dilworth-rca0}
Fix some~$k \in \NN$. For every coloring $f : [\NN]^2 \to 2$ which is transitive for color~0, and such that there is no $f$-homogeneous set for color~1 of some size~$k$, there is a coloring $g : \NN \to {k+1 \choose 2}$ such that for every~$i < {k+1 \choose 2}$, $g^{-1}(i)$ is $f$-homogeneous for color~0.
\end{proposition}
\begin{proof}[Felsner's algorithm~\cite{felsner1997online}]
Fix~$k$ and $f$.
We first construct a uniformly $f$-computable sequence $(F^1_s, \dots, F^{k-1}_s)_{s \in \NN}$ such that for every~$s \in \NN$,
\begin{itemize}
    \item[(1)] Each $F^i_s$ is a set of size at most~$i$, containing finite $f$-homogeneous sets for color~0;
    \item[(2)] $F^1_s \cup \dots \cup F^{k-1}_s$ forms a partition of $[0, s)$, that is, for every~$x < s$, there is exactly one~$i$ and one $H \in F^i_s$ such that $x \in H$;
    \item[(3)] For each~$F^i_s$, the set $\{ \max H : H \in F^i_s \}$ is $f$-homogeneous for color~1.
\end{itemize}
At stage~0, $F^1_0 = \dots = F^{k-1}_0 = \emptyset$.
At stage~$s+1$, assume $F^1_s, \dots F^{k-1}_s$ are defined.
Let $i \in \{1,\dots, k\}$ be the least index such that either $\card F^i_s < i$, or there is some~$H \in F^i_s$ such that $f(\max H, s) = 0$. Such an index exists as otherwise, by (3), $\{\max H : H \in F^{k-1}_s\} \cup \{s\}$ would form an $f$-homogeneous set for color~1 of size~$k$, contradicting our hypothesis. We have two cases:

Case 1: there is some~$H \in F^i_s$ such that $f(\max H, s) = 0$. Then let $\hat H = H \cup \{s\}$. If $i = 1$, then, letting $F^i_{s+1} = (F^i_s \setminus \{H\}) \cup \{\hat H\}$, and $F^j_{s+1} = F^j_s$ for $j \neq i$, the new tuple satisfies properties (1-3).
If $i > 1$, then, letting $F^{i-1}_{s+1} = F^i_s \setminus \{H\}$, $F^i_{s+1} = F^{i-1}_s \cup \{\hat H\}$, and $F^j_{s+1} = F^j_s$ for $j \not \in \{i-1,i\}$, the new tuple satisfies properties (1-3).

Case 2: there is no such~$H$. Then $\card F^i_s < i$ and, letting $F^i_{s+1} = F^i_s \cup \{ \{ s\}\}$ and $F^j_{s+1} = F^j_s$ for $j \neq i$, the new tuple satisfies properties (1-3). This completes the construction of the sequence.

Let $g : \NN \to {k+1 \choose 2}$ be defined inductively as follows: $g(0) = 0$. Given~$x$, let $g(x) = g(y)$ for any $y < x$ such that for some~$i \in \{1, \dots, k\}$ and some~$H \in F^j_{x+1}$, $x, y \in H$, if it exists. Otherwise, let $g(x)$ be the least fresh color. By construction and by (1), there are at most $\sum_{i \in \{1,\dots, k\}} i  = \frac{(1+k)k}{2} = {k+1 \choose 2}$ colors. Moreover, if $g(x) = g(y)$, then there is some $f$-homogeneous set~$H$ for color~0 such that $x, y \in H$, so $f(x, y) = 0$. It follows that for every $i < {k+1 \choose 2}$ $g^{-1}(i)$ is $f$-homogeneous for color~0.
\end{proof}

This simpler algorithm is sufficient to obtain another proof that $\EM$ implies $\BRT^2_2$ over~$\RCA_0$.

\begin{corollary}[{Soldà~\cite[Theorem 4.1.10]{solda2021calibrating}}]
$\RCA_0 \vdash \EM \to \BRT^2_2$.
\end{corollary}
\begin{proof}
Let $f : [\NN]^2 \to 2$ be a coloring with no $f$-homogeneous set for color~1 of size~$k$ for some~$k \in \NN$.
By $\EM$, there is an infinite subset~$X = \{ x_0 < x_1 < \dots \}$ which is $f$-transitive for both colors. 
Let $h : [\NN]^2 \to 2$ be defined by $h(a, b) = f(x_a, x_b)$. In particular, $h$ is transitive for both colors and has no $h$-homogeneous set for color~1 of size~$k$. By \Cref{prop:semi-transitive-dilworth-rca0}, there is some~$\ell \in \NN$ and a coloring $g : \NN \to \ell$ such that for every~$i < \ell$, $g^{-1}(i)$ is $h$-homogeneous for color~0. By $\RT^1$ which follows from $\EM$ (see Kreuzer~\cite{kreuzer2012primitive}), there is some~$i < \ell$ such that $g^{-1}(i)$ is infinite. The set $\{ x_a : g(a) = i \}$ is an infinite $f$-homogeneous set for color~0.
\end{proof}

\subsection{Mirsky's theorem}

Dilworth's theorem admits a dual version for partial orders with chains of bounded size, due to Mirsky~\cite{mirsky1971dual}.
The \emph{height} $h(\P)$ of a partial order $\P = (\NN, <_\P)$ is the size of its largest $\P$-chain.

\begin{statement}[Mirsky's theorem]
\qt{For every partial order $\P = (\NN, <_\P)$ of finite height, there is a coloring $g : \NN \to h(\P)$ such that for every~$i < h(\P)$, $g^{-1}(i)$ is a $\P$-antichain.}
\end{statement}

Accordingly, such a coloring $g$ is called an \emph{antichain partition} of~$\P$.  Contrary to Dilworth's theorem, Mirsky's theorem admits a trivial combinatorial proof: Given a partial order $\P = (\NN, <_\P)$ of finite height~$h$, define $g : \NN \to h$ as follows: $g(x)$ is the size of the largest $\P$-chain with~$x$ as $\P$-maximal element. However, from a reverse mathematical viewpoint, the definition of~$g$ requires an unbounded search for such a $\P$-chain, and is formalized over~$\ACA_0$. Hirst~\cite{hirst1987combinatorics} actually showed that Mirsky's theorem is equivalent to $\WKL$ (and therefore to Dilworth's theorem) over~$\RCA_0$, even when restricted to partial orders of height~2.

\begin{proposition}[{Hirst~\cite[Theorem 3.24]{hirst1987combinatorics}}]
Over $\RCA_0$, Mirsky's theorem is equivalent to~$\WKL$.
\end{proposition}

The proof of Mirsky's theorem over~$\RCA_0 + \WKL$ is very similar to that if Dilworth's theorem, this time using the comparability graph. As for Dilworth's theorem, one can define a weak version of the statement.

\begin{statement}[Weak Mirsky's theorem]
\qt{For every partial order $\P = (\NN, <_\P)$ of finite height, there is some~$k \in \NN$ and a coloring $g : \NN \to k$ such that for every~$i < k$, $g^{-1}(i)$ is a $\P$-antichain.}
\end{statement}

The natural combinatorial proof of Mirsky's theorem can be refined, to yield a proof of weak Mirsky's theorem over~$\RCA_0$.

\begin{proposition}\label[proposition]{prop:weak-mirsky-rca0}
Weak Mirsky's theorem is provable over~$\RCA_0$.
\end{proposition}
\begin{proof}
Let $\P = (\NN, <_\P)$ be a partial order of finite height~$h$.
Let $g_0 : \NN \to h$ be defined as follows: 
given $x \in \NN$, $g_0(x) = s$ where $x_0, \dots, x_s (=x)$ is the longest sequence of elements both $<_\P$-increasing and $<_\NN$-increasing.
The coloring $g_1 : \NN \to h$ is defined accordingly, where the sequence is $<_\P$-decreasing and $<_\NN$-increasing.
Last, let $g : \NN \to h^2$ be defined by $g(x) = \langle g_0(x), g_1(x) \rangle$.

We claim that for every~$i, j < h$, $g^{-1}(\langle i,  j \rangle)$ is a $\P$-antichain.
Indeed, otherwise, there is some~$x, y \in g^{-1}(\langle i,  j \rangle)$ with $x <_\NN y$ and $x <_\P y$ or $y <_\P x$. This contradicts the fact that $x, y \in g^{-1}_0(i)$ in the first case, and $x, y \in g^{-1}_1(j)$ in the second case.
\end{proof}

Mimouni and Patey~\cite{mimouni2025ramseylike} defined an asymmetric version of the Erd\H{o}s-Moser theorem in which the solution is required to be transitive for one color at least.

\begin{statement}[Half Erd\H{o}s-Moser theorem]
$\HEM$ is the statement \qt{For every coloring $f : [\NN]^2 \to 2$, there is an infinite $f$-transitive subset for some color.}
\end{statement}

By combining \Cref{prop:semi-transitive-dilworth-rca0} and \Cref{prop:weak-mirsky-rca0}, one obtains the following implication. It is unknown whether $\HEM$ implies $\BSig_2$ over~$\RCA_0$.

\begin{proposition}
    $\RCA_0 + \BSig_2 \vdash \HEM \to \BRT^2_2$.
\end{proposition}

\begin{proof}
Let $f : [\NN]^2 \to 2$ be a coloring with no $f$-homogeneous set for color~1 of size~$k$ for some~$k \in \NN$.
By $\HEM$, there is an infinite subset~$X = \{ x_0 < x_1 < \dots \}$ which is $f$-transitive for some color~$c < 2$.
Let $h : [\NN]^2 \to 2$ be defined by $h(a, b) = f(x_a, x_b)$. In particular, $h$ is transitive for color~$c$ and has no $h$-homogeneous set for color~1 of size~$k$. By \Cref{prop:semi-transitive-dilworth-rca0} if $c = 0$, and by \Cref{prop:weak-mirsky-rca0} if $c = 1$, there is some~$\ell \in \NN$ and a coloring $g : \NN \to \ell$ such that for every~$i < \ell$, $g^{-1}(i)$ is $h$-homogeneous for color~0. By $\RT^1$ which follows from $\BSig_2$, there is some~$i < \ell$ such that $g^{-i}$ is infinite. The set $\{ x_a : g(a) = i \}$ is an infinite $f$-homogeneous set for color~0.
\end{proof}

\section{$\BSig_2$ does not imply $\BRT^2_2$}\label[section]{sec:bsig2-not-urt22}

The goal of this section is to prove \Cref{main:bounded-non-comp-non-isigma2} and its consequences. 
The constructed instance~$f$ of $\BRT^2_2$ with no $\Delta_1(f)$-definable solution will also serve as the $\BRT^2_2$-instance witnessing that $\RCA_0 + \ADS \not \vdash \BRT^2_2$ in \Cref{sec:ads-not-urt22}.
Before constructing such an instance, let us introduce some basic notions about models of arithmetic with restricted induction, and the relation between failure of induction and cuts. See Hajek and Pudlak~\cite{hajek1998metamathematics} for a more extensive introduction.

\subsection{Models of weak arithmetic and cuts}\label[section]{sec:models-arithmetic}

A \emph{structure} in the language of second-order arithmetic is a tuple 
$$
\M = (M, S, 0, 1, +, \times, <)
$$
where $M$ denotes the set of integers, $S \subseteq \P(M)$ is the collection of sets of integers, $0$ and $1$ are constants in~$M$, and $+, \times : M^2 \to M$ are binary operations and $<$ is a binary relation on~$M$.
We shall exclusively work with structures which are models of~$\RCA_0$, and simply call them \emph{models}, and usually abbreviate them by $\M = (M, S)$.
A model $\M = (M, S)$ is \emph{topped} by a set~$Y \in S$ if every set in~$S$ is $\Delta_1(Y)$-definable with parameters in~$M$.

A model where $M$ is the set of standard integers $\omega = \{0, 1, \dots \}$ will be called an \emph{$\omega$-model}, if $M \neq \omega$, then the model is said to be \emph{non-standard}.

\begin{definition}[cut] 
    A \emph{cut} in a model $\M = (M, S, 0, 1, +, \times, <)$ is a nonempty subset $I \subseteq M$ which is closed by successor, and is an initial segment of $M$, that is, if $a \in I$ and $b \leq a$, then $b \in I$.
\end{definition}

In a non-standard model, the elements $0, 1, 1 + 1, 1 + 1 + 1, \dots$ form a cut which can be identified with $\omega$. The amount of induction satisfied by a model corresponds to the difficulty of having cuts that are definable.

\begin{proposition}
    Let $\M \models \RCA_0 + \BSig_n$ for some $n$, then $\M \models \neg \ISig_n$ if and only if there exists some $\Sigma_n^0$ formula $\phi(x)$, such that the set $\{x : \phi(x)\}$ forms a proper cut.
\end{proposition}

\noindent
\emph{Finite/infinite sets}. In models of weak arithmetic, the various definitions of infinity are not necessarily equivalent, and require a precise terminology.
A subset $X \subseteq M$ is said to be \emph{$M$-bounded} if there exists some $y \in M$ such that $x < y$ for every $x \in X$. If no such $y$ exists, $X$ is said to be \emph{$M$-unbounded}. In a model $\M = (M, S) \models \RCA_0$, the $M$-bounded sets belonging to $S$ will be called \emph{$\M$-finite} and the $M$-unbounded \emph{$\M$-infinite}.

Every element $a \in M$ can be seen as the binary encoding of some $M$-bounded set $F \subseteq M$ and we write $a = \sum_{x \in F} 2^x$. Not all $M$-bounded sets can be $M$-coded by some integer~$a \in M$, but if $\M = (M, S) \models \RCA_0$, then the $M$-coded sets are exactly the $\M$-finite, hence, the $\M$-finite sets only depends on the first order part $M$ of the model $\M$ and we will call them \emph{$M$-finite} to emphasize on that.
\smallskip

\noindent
\emph{Cofinal functions}. The lack of induction can also be measured in the existence of definable cofinal functions from a proper cut to~$M$. This will be of central importance for proving \Cref{main:bounded-non-comp-non-isigma2}.
For $I \subseteq M$ a cut, a function $g : I \to M$ is said to be \emph{cofinal (in $M$)} if its range is $M$-unbounded.

\begin{proposition}
    If $\M \models \RCA_0 + \BSig_n + \neg \ISig_n$ for some $n$, then there exists some increasing $\Sigma_n^0$-definable cofinal function $g : I \to M$ defined on some $\Sigma_n^0$ definable proper cut. 
\end{proposition}

In particular, if $\M = (M, S) \models \RCA_0 + \BSig_2 + \neg \ISig_2$ and $I \subseteq M$ is a $\Sigma^0_2$-cut bounded by some~$k \in M$, one can assume that the function~$g$ can be approximated by a $\Delta^0_1$ sequence of functions $(g_s)_{s \in M}$ such that for every~$s \in M$, $g_s : k \to M$ is increasing, and for every~$x < k$, $s \mapsto g_s(x)$ is non-decreasing, dominated by~$s \mapsto s$, and converges to~$g(x)$. One can also assume that for every $s \in M$ and $x < k$, if $g_{s+1}(x) \neq g_s(x)$ then $g_{s+1}(s) = s+1$. Hence, if $g(x) = s$ for some $x \in I$ and $s \in M$, then $g_s(x) = g(x)$.

\subsection{Main result}

In what follows, a $\Sigma_n(X)$-formula (respectively a $\Pi_n(X)$-formula or $\Delta_n(X)$-formula) is a $\Sigma_n^0$-formula (respectively a $\Pi_n^0$-formula or $\Delta_n^0$-formula) with only $X$ as a second-order parameter and any first-order parameters. 

The core of the argument lies in the following technical lemma:

\begin{lemma}[$\ISig_1$]\label[lemma]{lem:bounded-single-block}
    For every set~$X$, every $b \in \NN$ and $n \in \NN$, there exists a $\Delta_1(X)$ coloring $f_n^b : [\NN]^2 \to 2$ such that:
    \begin{itemize}
        \item There is no set of size $3$ that is $f_n^b$-homogeneous for the color $1$.
        \item For every $e < n$ such that $W^X_e$ is infinite, there exists some $x,y \in W^X_e$ such that $b \leq x < y$ and $f(x,y) = 1$.
    \end{itemize}
    Furthermore, the construction is uniform in $n$ and $b$.
\end{lemma}

\begin{proof}
The coloring $f_n^b$ will be defined by stages, with $f_n^b(x,s)$ being defined at stage $s$ for every $x < s$. At any point of the construction, any $x \in \NN$ will be in state \qt{0} or \qt{1}, meaning that at any further stage~$s$, $f(x, s)$ will be set accordingly. By default, every $x \in \NN$ has the state \qt{0}. We consider a trash $T_s \subseteq [0,s]$ where the elements leaving state \qt{1} will be put, to ensure that they will never go back to that state. Our construction will ensure that $\card T_s \leq n^2$ at every stage $s$.

For every $e < n$, we say that $e$ is \emph{active} at a stage $s$ if $\card (W^X_e[s] \cap [b,s]) > n^2$. Once an element is active, it stays active for the rest of the construction. \\

\textbf{Construction:} at stage $s$, if there exists some new $e < n$ that became active at this stage, we then put every $x < s$ that is in state \qt{1} in state \qt{0} and put them in the trash $T_s$. Then, for every $e' < n$ that is active, we pick an element in $W^X_{e'}[s] \cap [b,s] \setminus T_s$ an put it in state \qt{1}. If no such $e < n$ was found, then the all the elements stay in the same state and no new element is added to the trash.

Finally, whether or not a new $e$ was found, we let $f(x,s) = c$ for every $c < 2$ and every $x < s$ in state \qt{c}.   \\

\textbf{Claim 1}: $T_s \leq n^2$ for every $s \in \NN$, hence the construction is well-defined. Indeed, there are always at most $n$ elements in state \qt{1} at a given time (one per active $e$). And such elements are put into the trash at most $n$ times during the construction (once every time a new $e <n$ became active). Therefore, we can always pick an element in $W^X_{e'}[s] \cap [b,s] \setminus T_s$ if $W^X_{e'}$ is active at stage $s$.
\smallskip

\textbf{Claim 2}: There is no set of size $3$ that is $f_n^b$-homogeneous for the color $1$. Assume by contradiction that we have such a set $\{x < y < z\}$. By construction, $x$ was in state \qt{1} at stage $y$ for $f(x,y)$ to be equal to $1$. Then, for $f(y,z)$ to be equal to $1$, we need $y$ to be in state \qt{1} at stage $z$, but $y$ can only enter state \qt{1} after stage $y$, which would put $x$ in the trash, and make it so that $f(x,z) = 0$.
\smallskip

\textbf{Claim 3}: For every $e < n$ such that $W^X_e$ is infinite, there exists some $x,y \in W^X_e$ such that $b \leq x < y$ and $f(x,y) = 1$.
If $W^X_e$ is infinite, then by $\ISig_1$, there is some stage $s$ such that $\card(W^X_e[s] \cap [b,s]) > n^2$, hence $e$ is eventually active. Then, for every $y > s$, there will be some $x \in W^X_e$ that will be in state \qt{1} at stage $y$, which gives $f(x,y) = 1$. Since $W^X_e$ is infinite, there is one such $y$ in $W^X_e$.

\end{proof}

We are now ready to prove our first main theorem.

\begin{repmaintheorem}{main:bounded-non-comp-non-isigma2}
For every model $\M \models \RCA_0 + \BSig_2 + \neg \ISig_2$ and every set~$Z \in \M$, there is a $\BRT^2_2$-instance~$f \in \M$ with no $\Delta_1(f \oplus Z)$-definable solution. 
\end{repmaintheorem}

\begin{proof}
Let $\M = (M, S) \models \RCA_0 + \BSig_2 + \neg \ISig_2$ and $Z \in S$. In particular, there exists some~$X \in S$ and some proper $\Sigma_2(X)$-cut $I \subseteq M$ bounded by some~$k \in M$, and some cofinal $\Sigma_2(X)$ definable function  $g : I \to M$ with a $\Delta_1(X)$ approximation $(g_s)_{s \in M}$ (see \Cref{sec:models-arithmetic}). Note that the $g_s$ are uniformly $\Delta_1(X)$ in~$s$. As mentioned in \Cref{sec:models-arithmetic}, we assume that for every~$s \in M$, $g_s : k \to M$ is increasing, and for every $x \leq y < k$, $g_s(x) \leq g_s(y)$. Furthermore, for every $s \in M$ and $x < k$, $g_s(x) \leq s$ and if $g_{s+1}(x) \neq g_s(x)$ then $g_{s+1}(x) = s+1$.

For every $b,n \in \NN$, let $f_{n}^{b} : [M]^2 \to 2$ be the $\Delta_1(X \oplus Z)$-colorings defined in \Cref{lem:bounded-single-block}.
For every $i < k$, let $f^i : [M]^2 \to 2$ be defined by stages as follows: \\

\textbf{Construction:} At stage $s$: for every $i < k$, let $b^i_s \leq s$ be the last $b \leq s$ such that the value of $g_b(i)$ has changed, if it exists, otherwise, $b^i_s = 0$. Then let $f^i(x,s) = f_{g_s(i)}^{b^i_s}(x,s) = f^{b^i_s}_{b^i_s}(x,s)$ for every $x < s$ . \\

Finally, let $f : [M]^2 \to 2$ be defined by $f(x,y) = 1$ iff there exists some $i < k$ such that $f^i(x,y) = 1$ for every $x<y \in \NN$. \\

\textbf{Claim 1:} For every $i < k$, there is no $f^i$-homogeneous set for color~1 of size $3$. Notice first that for every $n,b \in \NN$, we have $f_{n}^{b}(x,y) = 0$ when $x < b$, therefore, if $f^i(x,y) = 1$ for some $x,y \in \NN$, then $b_i^y$ must be smaller than $x$ and therefore be equal to $b_i^x$. This means that every $f^i$-homogeneous set for color~1 if also homogeneous for some coloring $f_n^b$, hence that no such set of size three can exists, by our assumptions on the $f_n^b$.
\smallskip

\textbf{Claim 2:} There is some $\ell \in M$ such that there is no $f$-homogeneous set for color~1 of size~$\ell$. Since $\M \models \ISig_1$, the finite Ramsey's theorem for pairs and $k$ colors is true in $\M$. Let $\ell \in M$ be such that every finite coloring $h : [F]^2 \to k$ admits some $h$-homogeneous subset of size $3$ if $\card F \geq \ell$. If some set $F = \{x_0, \dots, x_{\ell - 1}\}$ was $f$-homogeneous for color~1, then, by definition of $f$, the coloring $h : [F]^2 \to k$ sending every pair $x_a,x_b$ to some $i < k$ such that $f^i(x_a,x_b) = 1$ would be well-defined, and by definition of $\ell$, there would exists some $i < k$ and some $f^i$-homogeneous set for the color~1 of size $3$, contradicting Claim 1. 
\smallskip

\textbf{Claim 3:} There is no $\M$-infinite $f$-homogeneous and $\Delta_1(X \oplus Z)$-set for color~0 in~$S$. Assume otherwise, and let $e \in M$ be such that $W^{X \oplus Z}_e$ is $M$-unbounded and $f$-homogeneous for color~0. By assumption on the $g_s$, the smallest stage $s \in \NN$ such that $g_s(i) = g(i)$ is $s = g(i)$. Pick some $x,y \in W^{X \oplus Z}_e$ such that $g(i) \leq x < y$ such that $f^{g(i)}_{g(i)}(x,y) = 1$, then, we have $f^i(x,y) = f^{g(i)}_{g_s(i)}(x,y) = 1$, hence $f(x,y) = 1$ and $W^{X \oplus Z}_e$ is not $f$-homogeneous for color~0.
\end{proof}

\begin{corollary}
$\RCA_0 + \BSig_2 \nvdash \BRT^2_2$.
\end{corollary}
\begin{proof}
Let $M$ be a countable model of $\mathsf{I}\Sigma_1  + \mathsf{B}\Sigma_2 +  \neg \mathsf{I}\Sigma_2$.
Let $\M = (M, S)$ be the model whose second-order part $S$ consists of the $\Delta_1$-definable sets with parameters in~$M$. By Friedman~\cite{friedman1976systems}, $\M \models \RCA_0 + \BSig_2$, but by \Cref{main:bounded-non-comp-non-isigma2}, $\M \not \models$ \qt{for every coloring $f : [\NN]^2 \to 2$ and every~$\ell \in \NN$ such that there is no $f$-homogeneous set for color~1 of size~$\ell$, there is an infinite $f$-homogeneous set for color~0.}
\end{proof}

\begin{corollary}
The statement \qt{$\BRT^2_2$ is computably true} is equivalent to $\ISig_2$ over $\RCA_0$.
\end{corollary}
\begin{proof}
Let $\M = (M, S) \models \RCA_0 + \ISig_2$ and let $f \in S$ be an instance of~$\BRT^2_2$.
By Friedman~\cite{friedman1976systems}, $\M_f = (M, \Definable{\Delta_1}{f}) \models \ISig_2$, where $\Definable{\Delta_1}{f}$ denotes the $\Delta_1(f)$-definable subsets of~$M$. By \Cref{prop:isig2-urt22}, $\M_f \models \BRT^2_2$, so there is some~$H \in \Definable{\Delta_1}{f}$ such that $\M_f \models $ \qt{$H$ is an $\BRT^2_2$-solution to~$f$}. In particular, $H \in S$, so $\M \models $ \qt{there is a $\Delta_1(f)$-definable $\BRT^2_2$-solution to~$f$.} Thus, $\M \models $ \qt{$\BRT^2_2$ is computably true}.

Suppose now that $\M \models \RCA_0 + $ \qt{$\BRT^2_2$ is computably true}. By Hirst~\cite{hirst1987combinatorics}, $\RCA_0 \vdash \BRT^2_2 \to \BSig_2$, so $\M \models \BSig_2$. Then, by \Cref{main:bounded-non-comp-non-isigma2}, $\M \models \ISig_2$.
\end{proof}

Note that the previous argument is very general: if $\ISig_2$ implies a $\Pi^1_2$-problem~$\Psf$ over $\RCA_0$, then $\RCA_0 + \ISig_2 \vdash$ \qt{$\Psf$ is computably true}. One could also have simply noticed that the proof of \Cref{prop:isig2-urt22} yields a computable solution.

\section{$\ADS$ does not imply $\BRT^2_2$}\label[section]{sec:ads-not-urt22}

The goal of this section is to prove \Cref{main:bsig2-not-isig2-ads-not-urt22} and derive its consequences, among which the fact that $\ADS$ does not imply $\BRT^2_2$ over~$\RCA_0$. The statement $\ADS$ was introduced by Hirschfeldt and Shore~\cite{hirschfeldt_combinatorial_2007} and later studied by many authors~\cite{manuel_lerman_separating_2013,patey_partial_2018,chong2021pi11}.

\begin{statement}[Ascending Descending Sequence]
$\ADS$ is the statement \qt{Every linear order $\L = (\NN, <_\L)$ admits an infinite $\L$-ascending or $\L$-descending sequence.}
\end{statement}

Following the standard decomposition of $\RT^2_2$ into its stable and cohesive version, Hirschfeldt and Shore~\cite{hirschfeldt_combinatorial_2007} proved that $\ADS$ is equivalent to $\SADS + \COH$ over $\RCA_0$, where $\SADS$ and $\COH$ are defined below.
Given a linear order~$\L = (\NN, <_\L)$, an element~$x \in \NN$ is \emph{small} (\emph{large}) if for all but finitely many~$y \in \NN$, $x <_\L y$ ($x >_\L y$).
A linear order is \emph{stable} if every element is either small or large.

\begin{statement}[Stable Ascending Descending Sequence]
$\SADS$ is the statement \qt{Every stable linear order $\L = (\NN, <_\L)$ admits an infinite $\L$-ascending or $\L$-descending sequence.}\footnote{The statement $\SADS$ was originally defined by Hirschfeldt and Shore~\cite{hirschfeldt_combinatorial_2007} only for linear orders of type $\omega+\omega^*$. Indeed, the remaining stable order types $\omega+k$ and $k+\omega^*$ admit trivial solutions over~$\RCA_0$. The two definitions are therefore equivalent over~$\RCA_0$.}
\end{statement}

Given an infinite sequence of sets $\vec{R} = R_0, R_1, \dots$, an infinite set~$C$ is \emph{$\vec{R}$-cohesive}
if for every~$n \in \NN$, $C \subseteq^* R_n$ or $C \subseteq^* \overline{R}_n$, where $\subseteq^*$ means \emph{almost included}.

\begin{statement}[Cohesiveness]
$\COH$ is the statement \qt{Every countable sequence of sets admits an infinite cohesive set.}
\end{statement}

Over $\RCA_0 + \BSig_2$, the cohesiveness principle can be better understood as a $\Delta^0_2$-version of weak K\"onig's lemma, that is, the statement \qt{For every $\Delta^0_2$ infinite binary tree~$T \subseteq 2^{<\NN}$, there is a $\Delta^0_2$-path $P \in [T]$} (see Jockusch and Stephan~\cite{jockusch1993cohesive} for a computability-theoretic version, and Belanger~\cite{belanger_conservation_2015} for a formalized version).

\subsection{Strategy to prove \Cref{main:bsig2-not-isig2-ads-not-urt22}}

Before going into the technicalities, let us first outline the proof of \Cref{main:bsig2-not-isig2-ads-not-urt22}.
It is divided into the following steps:
\smallskip

\noindent
\emph{Step 1.}
Consider any countable model~$\M_0 = (M, S_0)$ of $\RCA_0 + \BSig_2 + \neg \ISig_2$ topped by some set~$Z_0$. By \Cref{main:bounded-non-comp-non-isigma2}, $\M_0$ contains a stable instance~$f$ of $\BRT^2_2$ with no $\Delta_1(Z_0)$-definable solution. Let $A = \{ x \in M : \forall^\infty y f(x, y) = 0 \}$. The set $A$ is $\Definable{\Delta_2}{\M_0}$ and $Z_0$-immune. Actually, the proof of \Cref{main:bounded-non-comp-non-isigma2} shows that $A$ satisfies a stronger notion of immunity, that we call for now \emph{sufficient $Z_0$-immunity}. 
\smallskip

\noindent
\emph{Step 2.} Given a countable model~$\M_n = (M, S_n)$ of $\RCA_0 + \BSig_2$ topped by a set~$Z_n$, and any sufficiently $Z_n$-immune set~$A \subseteq M$, we shall prove that for every instance $\L$ of $\SADS$ in~$\M_n$, there is a solution~$G$ such that $A$ is sufficiently $G \oplus Z_n$-immune, and $\M_n[G] \models \RCA_0 + \BSig_2$. Then, let $\M_{n+1} = \M_n[G]$ be the $\omega$-extension of~$\M_n$ topped by $Z_{n+1} = G \oplus Z_n$ (\Cref{prop:sads-eff}).
\smallskip

\noindent
\emph{Step 3.} Given a countable model~$\M_n = (M, S_n)$ of $\RCA_0 + \BSig_2$ topped by a set~$Z_n$, and any sufficiently $Z_n$-immune set~$A \subseteq M$, we shall prove that for every instance $\vec{R}$ of $\COH$ in~$\M_n$, there is a solution~$G$ such that $A$ is sufficiently $G \oplus Z_n$-immune, and $\M_n[G] \models \RCA_0 + \BSig_2$. Then, let $\M_{n+1} = \M_n[G]$ be the $\omega$-extension of~$\M_n$ topped by $Z_{n+1} = G \oplus Z_n$ (\Cref{prop:coh-eff}).
\smallskip

\noindent
\emph{Step 4.} Apply iteratively Step~2 and Step~3 to obtain an infinite sequence of countable models $\M_0, \M_1, \dots$ of $\RCA_0 + \BSig_2$ such that for each~$n \in \omega$, $\M_{n+1}$ is topped by $Z_n$, $\omega$-extends $\M_n$ and $A$ is sufficiently $Z_n$-immune, where~$A$ is the $\Delta_2(\M_0)$ set defined in Step~1. A union of models of $\RCA_0 + \BSig_2$ is again a model of $\RCA_0 + \BSig_2$, and if the instances of $\SADS$ and $\COH$ are properly listed, then $\Nc = \bigcup_n \M_n$ is a model of $\SADS + \COH$, hence of $\ADS$. Moreover, $A$ is sufficiently immune relative to every set in~$\Nc$, so the instance of~$\BRT^2_2$ defined in Step~1 has no solution in~$\Nc$, hence $\Nc \not \models \BRT^2_2$.

\subsection{Combinatorics of $\ADS$}\label[section]{sec:combinatorics-ads}

In order to get a better grasp of the appropriate notion of immunity to consider, let us dig into the combinatorics of~$\ADS$. There exists mainly two forcing constructions to build solutions to $\ADS$. The first one is asymmetric, and was used by Hirschfeldt and Shore~\cite[Proposition 2.26]{hirschfeldt_combinatorial_2007} to prove that $\ADS$ does not imply the existence of DNC functions over $\RCA_0$. A function $g : \NN \to \NN$ is \emph{diagonally non-$X$-computable} ($X$-DNC) if for every~$e \in \NN$, $g(e) \neq \Phi^X_e(e)$. The second forcing construction is symmetric, and was defined by Lerman, Solomon and Towsner~\cite{manuel_lerman_separating_2013} and later used by Patey~\cite{patey_partial_2018} to give an alternative proof of \cite[Proposition 2.26]{hirschfeldt_combinatorial_2007}.

The combinatorics of $\ADS$ are already witnessed by stable linear orders, so let us fix a computable stable linear order $\L = (\NN, <_\NN)$. If it is of order type $\omega+k$ or $k+\omega^*$, then there is an infinite computable $\L$-ascending or $\L$-descending sequence, so let us assume that it is of order type $\omega+\omega^*$. Let $U^0$ and $U^1$ be its sets of $\L$-small and $\L$-large elements, respectively. Note that $U^0 \sqcup U^1 = \NN$ and both sets are $\Delta^0_2$.

The natural notion of forcing for building solutions to~$\L$ is of disjunctive nature. A \emph{condition} is a pair $(\sigma^0, \sigma^1)$ such that $\sigma^0$ and $\sigma^1$ are finite $\L$-ascending and $\L$-descending sequences, respectively. To ensure extensibility of each initial segment, one furthermore requires that $\sigma^0 \subseteq U^0$ and $\sigma^1 \subseteq U^1$. Because of this restriction, being a condition is a $\Delta^0_2$ predicate. A condition $(\tau^0, \tau^1)$ \emph{extends} $(\sigma^0, \sigma^1)$ if $\sigma^i \preceq \tau^i$ for both~$i < 2$.

Thankfully, one can define a computable notion of \emph{split pair}, which is a pair $(\sigma^0, \sigma^1)$ such that $\sigma^0$ is $\L$-ascending, $\sigma^1$ is $\L$-descending, and $\max_\L \sigma^0 <_\L \min_\L \sigma^1$. The core of the combinatorics of $\ADS$ lies in the following fact:

\begin{lemma}
If $(\sigma^0, \sigma^1)$ is a condition and $(\tau^0, \tau^1)$ is a split pair such that $\sigma^i \preceq \tau^i$ for both~$i < 2$, then either $(\tau^0, \sigma^1)$ or $(\sigma^0, \tau^1)$ is a valid condition extending $(\sigma^0, \sigma^1)$.
\end{lemma}
\begin{proof}
Since $\max_\L \sigma^0 <_\L \min_\L \sigma^1$, either $\tau^0 \subseteq U^0$ or $\tau^1 \subseteq U^1$.
\end{proof}

Let us explain how the combinatorics above enable to have a short c.e.\ description of large subsets of $G$-c.e.\ sets, where~$G$ is the constructed solution. Fix a condition $p = (\sigma^0, \sigma^1)$ and two Turing indices $e_0, e_1 \in \NN$. Let $a_0, a_1 : \NN \to \NN$ be defined as follows: for every~$t \in \NN$, $W_{a_0(t)}$ and $W_{a_1(t)}$ search for a split pair $(\tau^0, \tau^1)$ below $p$ such that $\card W^{\tau^0}_{e_0} \geq t$ and $\card W^{\tau^1}_{e_1} \geq t$. If such a split pair exists, then $W_{a_0(t)} = W^{\tau^0}_{e_0}$ and $W_{a_1(t)} = W^{\tau^1}_{e_1}$. 
By the combinatorics above, for every~$t$, if such a pair exists, then there is an extension~$q$ of~$p$ forcing either $W_{a_0(t)} \subseteq W^{G^0}_{e_0}$ or $W_{a_1(t)} \subseteq W^{G^1}_{e_1}$. Otherwise, there is an extension~$q$ of~$p$ forcing either $\card W^{G^0}_{e_0} < t$ or $\card W^{G^1}_{e_1} < t$. Last, note that $p$, $e_0$ and $e_1$ being fixed, the functions $a_0$ and $a_1$ have a growth in~$\mathcal{O}(t)$ for an appropriate numbering.

Such combinatorics are particularly useful when dealing with preservation of effective notions of immunity.

\begin{definition}
Fix a function~$h : \NN \to \NN$ and a set~$Z$.
A set $A$ is said to be \emph{$h$-effectively $Z$-immune} if for every Turing index $e \in \NN$, $\card W_e^Z \geq h(e)$ implies $W_e^Z \not \subseteq A$.
\end{definition}

We simply say that~$A$ is \emph{effectively $Z$-immune} if it is $h$-effectively immune for a $Z$-computable function~$h$.
Arslanov, Nadirov, and Solov'ev~\cite{arslanov1977completeness} proved that every effectively immune set computes a DNC function, and conversely, Jockusch~\cite{jockusch1989fpf} proved that every DNC function computes an effectively immune set, witnessed by the identity function. Let $A$ be such an effectively immune set. By slightly modifying the above argument, using Kleene's fixpoint theorem, one can easily obtain indices $a_0$ and $a_1$ such that $W_{a_0}$ and $W_{a_1}$ search for a split pair  $(\tau^0, \tau^1)$ below $p$ such that $\card W^{\tau^0}_{e_0} \geq a_0$ and $\card W^{\tau^1}_{e_1} \geq a_1$. Then, if such a split pair exists, there is an extension $q$ forcing $W_{a_0} \subseteq W^{G^0}_{e_0}$ or $W_{a_1} \subseteq W^{G^1}_{e_1}$, hence forcing $W^{G^0}_{e_0} \not \subseteq A$ or $W^{G^1}_{e_1} \not \subseteq A$.

In our setting, the appropriate notion of immunity will be the following variant of effective immunity. Let $(n, a) \mapsto 2_n(a)$ be the tetration function defined inductively as follows: $2_0(a) = a$ and $2_{n+1}(a) = 2^{2_n(a)}$. 

\begin{definition}
Fix a set $Z$, an unbounded $\Delta_2(Z)$ set~$B$ and some standard $n \in \omega$. A set~$A \subseteq M$ is \emph{$B$-block $Z$-immune with cost~$n$} if for every~$b \in B$, and every $x < b$, $\card W^Z_x \geq 2_n(b)$ implies $W^Z_x \not \subseteq A$.
\end{definition}

\subsection{Effective constructions over $\BSig_2 + \neg \ISig_2$}

The proofs of \Cref{prop:sads-eff,prop:coh-eff} are done by effectivizations of forcing constructions over non-standard models. By an isomorphism theorem of Fiori-Carones, Kołodziejczyk, Wong, and Yokoyama~\cite{fiori2021isomorphism}, there are some deep reasons why, given a $\Pi^1_2$-problem~$\Psf$, any $\omega$-extension of a model of~$\RCA_0 + \BSig_2 + \neg \ISig_2$ into a model of~$\RCA_0 + \BSig_2 + \Psf$ should be through a first-jump control construction effective in any PA degree over $\emptyset'$. This will be in particular the case in our construction.

Effective constructions in models of arithmetic with restricted induction raise some new issues.
Indeed, given a countable model~$\M = (M, S)$, suppose one builds a set~$G \subseteq M$ satisfying a countable sequence of requirements $(\R_e)_{e \in M}$ using an effectivized forcing argument. Let $(\PP, \leq)$ be the associated notion of forcing. In order to effectivize the argument, each condition~$p \in \PP$ is given a code $\langle p \rangle \in \NN$ and the decreasing sequence of conditions must be represented by a $\emptyset'$-effective decreasing sequence of codes.

In the standard setting, that is, $M = \omega$, one can consider each requirement $\R_e$ individually, prove a density lemma saying that every condition~$p$ has an extension~$q$ forcing~$\R_e$, and then build an infinite descending sequence of conditions $p_0 \geq p_1 \geq \dots$ such that $p_{n+1}$ forces $\R_n$. On the other hand, if $\M \models \BSig_2 + \neg \ISig_2$, it might be the case that the decreasing sequence is defined only on a proper $\Sigma^0_2$-cut~$I \subseteq M$, that is, only $(p_i)_{i \in I}$ is defined and only the requirements $(\R_i)_{i \in I}$ are satisfied. Note that in this case, the sequence of codes $(\langle p_i\rangle)_{i \in I}$ is unbounded in~$M$, otherwise the cut~$I$ would admit a $\Delta^0_2$-description, contradicting the fact that $\M \models \BSig_2$.

To overcome this issue, instead of proving a density lemma for each requirement $\R_e$ individually, we will prove a stronger lemma, stating that for every~$k \in M$ and every condition~$p$, there is an extension~$q$ forcing $\R_e$ for every~$e < k$ simultaneously. Then, using this density lemma, build a decreasing sequence of conditions $p_0 \geq p_1 \geq \dots$ such that $p_{i+1}$ forces $\R_e$ for every~$e < \langle p_i\rangle$ simultaneously. Even if the sequence is only defined on a proper cut~$I$, all the requirements $(\R_e)_{e \in M}$ will be forced. This technique is known as \emph{Shore blocking}.

\subsection{The $\SADS$ step}

We now prove the existence of $\omega$-extensions which add solutions to instances of~$\SADS$ while preserving the fact that the instance of~$\BRT^2_2$ has no solution. 

\begin{proposition}\label[proposition]{prop:sads-eff}
    Let $\M = (M,S) \models \RCA_0 + \BSig_2$ be a countable model topped by some set $Y$. Let~$B$ be an $M$-unbounded $\Delta_2(Y)$ set and fix $n \in \omega$. Then, for every set $A \subseteq M$ that is $B$-block $Y$-immune with cost~$n$ and every instance $\L = (M, <_\L)$ of $\SADS$ in $S$, there exists some $M$-unbounded set $G \subseteq M$ and some $M$-unbounded $\Delta_2(Y)$ set $\hat{B}$ such that:
    \begin{itemize}
        \item $G$ is $\L$-monotonous 
        \item $Y' \geq_T (G \oplus Y)'$ and in particular $\M[G] \models \RCA_0 + \BSig_2$
        \item $A$ is $\hat{B}$-block $(G \oplus Y)$-immune with cost $n+2$.
    \end{itemize}
\end{proposition}

\begin{proof}
The following proof is adapted from \cite[Proposition 3.17]{houerou2025conservation}.

Let $\L = (M, \leq_\L)$ be an instance of $\SADS$ in $S$, let $U^0 \subseteq M$ and $U^1 \subseteq M$ be the sets of elements that have a finite amount of predecessors and successors, respectively. By assumption, $U^0 \sqcup U^1 = M$, and $U^0$ and $U^1$ are both $\Delta_2(Y)$. If either $U^0$ or $U^1$ is $M$-bounded, then there exists some $\M$-infinite $Y$-computable $\L$-monotonous sequence and we can take $G$ to be such a sequence, hence, from now on, both $U^0$ or $U^1$ will be assumed to be $M$-unbounded. 

For simplicity of notation, the elements $b \in B$ will all be assumed to be of the form $2^{2^k}$ for some $k \in M$, which will then be used to construct $\hat{B}$ such that $\hat{B} \subseteq \{\log_2(\log_2(b)) : b \in B\}$. The general case requires some care in handling integer parts and inequalities, but does not change the argument. 

We will define two sets $G^0$ or $G^1$ using an infinite $\Delta_1(Y')$ decreasing sequence of conditions and then pick $G$ to be equal to one of them. We will also define two $\Delta_1(Y')$ sets $B^i$ for $i < 2$ so that, if $G = G^i$, then $A$ is $B^i$-block $(G \oplus Y)$-immune with cost $n+2$.

\begin{definition}
A \emph{condition} is a pair $(\sigma^0, \sigma^1)$ with $\sigma^0$ an $M$-finite $\L$-ascending sequence included in $U^0$ and $\sigma^1$ an $M$-finite $\L$-descending sequence included in $U^1$. We have $(\tau^0, \tau^1) \leq (\sigma^0, \sigma^1)$ if $\tau^0 \succeq \sigma^0$ and $\tau^1 \succeq \sigma^1$.
\end{definition}

Being a condition is a $\Delta_2(Y)$-predicate, while the extension relation is $\Delta_1$. We define a weaker notion of pair which is $\Delta_1(Y)$:
\begin{definition}
A \emph{split pair} is a pair $(\sigma^0, \sigma^1)$ of strings such that $\sigma^0$ is an $M$-finite $\L$-ascending sequence, $\sigma^1$ is an $M$-finite $\L$-descending sequence, and $\max_\L \sigma^0 <_\L \min_\L \sigma^1$.
\end{definition}

As mentioned in \Cref{sec:combinatorics-ads}, the core property of split pairs $(\sigma^0, \sigma^1)$ comes from the fact that if $\sigma^i \not \subseteq U^i$ for some $i < 2$, then $\sigma^{1-i} \subseteq U^{1-i}$. In particular, if $(\sigma^0, \sigma^1)$ is a condition and $(\tau^0, \tau^1)$ is a split pair such that $\sigma^0 \preceq \tau^0$ and $\sigma^1 \preceq \tau^1$, then either $(\tau^0, \sigma^1)$, or $(\sigma^0, \tau^1)$ is a valid condition extending $(\sigma^0, \sigma^1)$. Given two split pairs $(\sigma^0, \sigma^1)$ and $(\tau^0, \tau^1)$, we write $(\tau^0, \tau^1) \leq (\sigma^0, \sigma^1)$ if $\sigma^i \preceq \tau^i$ for each~$i < 2$.

\begin{definition}
Let $(\sigma^0, \sigma^1)$ be a condition, $e, x \in M$ and $i < 2$, we write:
\begin{itemize}
    \item $(\sigma^0, \sigma^1) \Vdash_i \Phi_e^{G^i \oplus Y}(x) \downarrow$ if $\Phi_e^{\sigma^i \oplus Y}(x) \downarrow$. 
    \item $(\sigma^0, \sigma^1) \Vdash_i \Phi_e^{G^i \oplus Y}(x) \uparrow$ if for every split pair $(\tau^0, \tau^1) \leq (\sigma^0, \sigma^1)$ with $\tau^{1-i} = \sigma^{1-i}$, we have $\Phi_e^{\tau^i \oplus Y}(x)\uparrow$. 
\end{itemize}   
\end{definition}

The relation $(\sigma^0, \sigma^1) \Vdash_i \Phi_e^{G^i \oplus Y}(x) \downarrow$ is $\Delta_1(Y)$ and the relation $(\sigma^0, \sigma^1) \Vdash_i \Phi_e^{G^i \oplus Y}(x) \uparrow$ is $\Pi_1(Y)$, both uniformly in their parameters. The formulas $\Phi_e^{G^i \oplus Y}(x) \downarrow$ and $\Phi_e^{G^i \oplus Y}(x) \uparrow$ are universally $\Sigma_1(G^i \oplus Y)$ and $\Pi_1(G^i \oplus Y)$ respectively and the above definition induces a forcing relation on these classes of formulas.

We want our construction to force for some~$i < 2$ the following requirements for every~$k \in M$:
\begin{itemize}
    \item $\mathcal{R}^i_k$: $(G^i \oplus Y)' \uh_k$ is decided, that is, there exists some $\sigma' \in 2^k$ such that $(G^i \oplus Y)' \uh_k = \sigma'$.
    \item $\mathcal{S}^i_{k}$: there exists some $\ell > 2^{2^k}$ in $B$, such that, for every $x < \log_2(\log_2(\ell))$, if $\card W^{G^i \oplus Y}_x \geq 2_n(\ell)$, then, there exists some $e < \ell$ such that $\card W_e^Y \geq 2_n(\ell)$ and $W_e^Y \subseteq W_x^{G^i \oplus Y}$ (in particular, since  $A$ is $B$-block $Y$-immune with cost $n$, this implies that $W_e^Y \not \subseteq A$, hence that $W^{G^i \oplus Y}_x \not \subseteq A$).
    \item $\mathcal{T}^i_k$: $\exists x > k, x \in G^i$
\end{itemize}
For this, we will use a pairing argument, and ensure that for every~$k \in M$, the following requirements are forced: $\R^0_k \vee \R^1_k$, $\R^0_k \vee \S^1_k$, $\S^0_k \vee \R^1_k$, $\S^0_k \vee \S^1_k$, and $\T^0_k \wedge \T^1_k$. 

In what follows, $(\sigma^0, \sigma^1) \Vdash_i \R_k^i$ means that there exists some $\sigma' \in 2^k$ such that $(\sigma^0, \sigma^1) \Vdash_i \Phi_e^{G^i \oplus Y}(e) \downarrow$ for every $e < k$ such that $\sigma'(e) = 1$ and $(\sigma^0, \sigma^1) \Vdash_i \Phi_e^{G^i \oplus Y}(e) \uparrow$ for every $e < k$ such that $\sigma'(e) = 0$. Similarly $(\sigma^0, \sigma^1) \Vdash_i \S_k^i$ means that there exists some $\ell > 2^{2^k}$ in $B$ such that for every $x < \log_2(\log_2(\ell))$, either $(\sigma^0, \sigma^1)$ forces $\card W_x^{G^i \oplus Y} < 2_n(\ell)$ or forces $\card W_x^{G^i \oplus Y} \geq 2_n(\ell)$, and in that case $(\sigma^0, \sigma^1)$ also forces $W_e^Y \subseteq W_x^{G^i \oplus Y}$ for some $e < \ell$ such that $\card W_e^Y \geq 2_n(\ell)$. The requirement $\T^1_k$ is already a $\Sigma_1(G)$ formula, and the forcing relation has already been defined for such formulas.

\begin{definition}
Fix a condition $(\sigma^0, \sigma^1)$.
A \emph{$(\sigma^0, \sigma^1)$-tree} is a binary tree $T \subseteq 2^{<M}$ labelled by a family of split pairs $(\sigma^0_{\rho}, \sigma^1_{\rho})_{\rho \in T}$ such that for all $\rho \in T$: 
\begin{itemize}
    \item $\sigma^0 \preceq \sigma^0_{\rho}$ and $\sigma^1 \preceq \sigma^1_{\rho}$ ;
    \item For every $\mu \in T$ and every $i < 2$, if $\rho \cdot i \preceq \mu$ then $\sigma^i_{\rho} \preceq \sigma^i_{\mu}$ ;
\end{itemize}
     A node $\rho$ is said to be \emph{valid} if for every~$i < 2$ and every $\mu \in T$ such that $\mu \cdot i \preceq \rho$, then $\sigma^i_\mu \subseteq U^i$. 
     Being valid is a $\Delta_2(Y)$ condition.
\end{definition}

\begin{definition}
    Let $T$ a $(\sigma^0,\sigma^1)$-tree 
    \begin{itemize}
        \item A node $(\sigma^0_{\rho}, \sigma^1_{\rho})$ in $T$ \emph{satisfies} a pair of $\Sigma_1(Y)$-formulas $(\varphi_0(G^0, x), \varphi_1(G^1,x))$ if $\varphi_0(\sigma^0_{\rho}, |\rho| - \card \rho)$ and $\varphi_1(\sigma^1_{\rho}, \card \rho)$ holds (where $\card \rho$ is the cardinal of $\rho$ seen as a set, to be distinguished from the length~$|\rho|$ of~$\rho$ as a string). 
        \item $T$ \emph{satisfies} $(\varphi_0(G^0, x), \varphi_1(G^1,x))$ if every node of $T$ satisfies it.
    \end{itemize}
\end{definition}

\begin{lemma}\label[lemma]{lem:sads-exists-maximal-tree}
For every pair $(\varphi_0(G^0, x), \varphi_1(G^1,x))$ of $\Sigma_1(Y)$ formulas and every condition $(\sigma^0, \sigma^1)$, there is a maximal $\Sigma_1(Y)$ $(\sigma^0, \sigma^1)$-tree~$T$ satisfying it.
\end{lemma}

\begin{proof}
    Same as \cite[Lemma 3.22]{houerou2025conservation}.
\end{proof}

The tree $T$ obtained in \cite[Lemma 3.22]{houerou2025conservation} is given by the limit of a uniformly $Y$-computable increasing sequence $(T_s)_{s \in M}$ of $M$-finite $(\sigma^0, \sigma^1)$-trees. This sequence can be uniformly $Y$-computed in the pair $(\varphi_0(G^0, x), \varphi_1(G^1,x))$ and in the condition $(\sigma^0, \sigma^1)$. Thus, there exists some computable function $e : M \to M$ such that $W^Y_{e(\langle \ulcorner\varphi_0\urcorner, \ulcorner\varphi_1\urcorner, \sigma^0, \sigma^1, \rho ,i,x\rangle)}$ searches the split pair $(\sigma^0_{\rho}, \sigma^1_{\rho})$ in position $\rho$ of the corresponding $(\sigma^0, \sigma^1)$-tree $T$, and, if such a split pair exists, outputs the elements of $W_{x}^{\sigma^i_{\rho} \oplus Y}$.

\begin{lemma}\label[lemma]{lem:sads-finite-tree-forcing}
Let $(\sigma^0, \sigma^1)$ be a condition, $(\varphi_0(G^0, x), \varphi_1(G^1,x))$ be a pair of $\Sigma_1(Y)$ formulas, and $T$ be an $M$-finite maximal $(\sigma^0, \sigma^1)$-tree satisfying $(\varphi_0(G^0, x),\allowbreak \varphi_1(G^1,x))$.
Then there is some $i < 2$, some extension $(\hat \sigma^0, \hat \sigma^1) \leq (\sigma^0, \sigma^1)$, some~$p_i\in M \cup \{-1\}$ such that the following holds:
\begin{itemize}
    \item $(\hat \sigma^0, \hat \sigma^1) \Vdash_i \varphi_i(G^i, p_i)$
    \item $(\hat \sigma^0, \hat \sigma^1) \Vdash_i \neg \varphi_i(G^i,p_i+1)$.
\end{itemize}
With the convention that $(\hat \sigma^0, \hat \sigma^1) \Vdash_i \varphi_i(G^i,-1)$ always holds.

The side $i < 2$ and the extension $(\hat \sigma^0, \hat \sigma^1)$ is $Y'$-computable from $(\sigma^0, \sigma^1)$. Furthermore, $\hat \sigma^i$ is an element of some split pair appearing in $T$, i.e., there exists some $\rho \in T$ such that $\hat \sigma^i = \sigma^i_{\rho}$.
\end{lemma}

\begin{proof}
    Same as Lemma 3.23 in the arXiv version\footnote{The published version of \cite{houerou2025conservation} contains a small flaw in the proof of conservation of SADS, fixed in the arXiv version (\url{https://arxiv.org/abs/2402.11616}) using Lemma 3.23.} of \cite{houerou2025conservation}.
\end{proof}

\begin{lemma}\label{lem:sads-forcing-eff-jump}
Let $k \in M$ and $(\sigma^0, \sigma^1)$ be a condition. Then, for every requirements $\mathcal{U}^0_k \in \{\R^0_k, \S^0_k\}$ and $\mathcal{U}^1_k \in \{\R^1_k, \S^1_k\}$, there exists some extension $(\hat \sigma^0, \hat \sigma^1) \leq (\sigma^0, \sigma^1)$ and some $i < 2$ such that $(\hat \sigma^0, \hat \sigma^1) \Vdash_i \mathcal{U}^i_k$. Such an extension is uniformly $Y'$-computable in $\U_k^0, \U_k^1$ and $(\sigma^0, \sigma^1)$.
\end{lemma}

\begin{proof}

We will define two $\Sigma_1(Y)$-formulas $\Phi_t(G,x)$ and $\Psi_t(G,x)$ as follows:


Let $\Phi_{t}(G,x)$ be the formula: $\exists e_0 < e_1 < \dots < e_{x} < t$ such that $\Phi_{e_i}^{G \oplus Y}(e_i) \downarrow$ for every $i \leq x$, and let $\Psi_t(G,x)$ be the formula: $\exists e_0 < e_1 < \dots < e_{x} < t$ such that $\card W_{e_i}^{G \oplus Y} \geq 2_{n+2}(t)$ for every $i \leq x$.

Consider the function $S : M \to M$ which maps $t$ to the biggest value taken by $e(\ulcorner\varphi_0 \urcorner, \ulcorner \varphi_1 \urcorner, \sigma^0, \sigma^1, \rho, i, x)$ for every $\varphi_0, \varphi_1 \in \{\Phi_t, \Psi_t \}$, every $\rho \in 2^{\leq 2(t-1)}$, every $i < 2$ and every $x < t$. 

Assuming a reasonable encoding of $e$ and of $\Phi_t$ and $\Psi_t$, this function grows slower than $t \mapsto 2^{2^t}$, indeed, all the possible parameters of $e$ either directly depend on $t$ (such as $\ulcorner\varphi_0 \urcorner$ and $\ulcorner \varphi_1 \urcorner$), are smaller than $t$ (such as $x$) or are in $2^{\leq2(t-1)}$ (like $\rho$), hence, in any case they can be encoded using $\alpha t$ bits for some $\alpha \in \omega$ and the same holds for the corresponding value taken by $e$. Hence, there exists some element $\ell$ in $B$ such that $\ell>2^{2^{k}}$ and such that $\ell' := \log_2(\log_2(\ell))$ satisfies $S(\ell') < 2^{2^{\ell'}} = \ell$.

For every $i < 2$, let $\varphi_i(G,x)$ be the formula $\Phi_{\ell'}(G,x)$ if $\mathcal{U}^i_k = \R^i_k$ and $\Psi_{\ell'}(G,x)$ otherwise. \\

Let $T$ be the maximal $\Sigma_1(Y)$ $(\sigma^0, \sigma^1)$-tree~$T$ satisfying $(\varphi_0(G, x), \varphi_1(G,x))$ obtained in \Cref{lem:sads-exists-maximal-tree}. Since $\Phi_{\ell'}(G,\ell')$ and $\Psi_{\ell'}(G,\ell')$ cannot hold, as there are only $\ell'$ Turing indices smaller than $\ell'$, the tree $T$ cannot have a depth bigger than $2(\ell'-1)$. Indeed, any $\rho \in T$ of length $> 2(\ell'-1)$ would make either $|\rho| - \card \rho$ or $\card \rho$ bigger than $\ell'$, contradicting the fact that $T$ satisfies $(\varphi_0(G, x), \varphi_1(G,x))$.

Therefore, $T$, as a binary tree, has less than $2^{2\ell' - 1}$ elements, and the increasing sequence $(T_s)_{s \in M}$ of computable trees approaching $T$ must have a final stage after which the approximation stabilizes and is equal to $T$, hence $T$ is $M$-finite. \\

Therefore, by \Cref{lem:sads-finite-tree-forcing}, there exists some extension $(\hat \sigma^0, \hat \sigma^1) \leq (\sigma^0, \sigma^1)$, some $i < 2$ and some $p_i \in \{-1,0, \dots,\ell'-1\}$  such that:
\begin{itemize}
    \item $(\hat \sigma^0, \hat \sigma^1) \Vdash_i \varphi_i(G^i, p_i)$
    \item $(\hat \sigma^0, \hat \sigma^1) \Vdash_i \neg \varphi_i(G^i,p_i + 1)$.
\end{itemize}

There are two cases:
\smallskip

\textbf{Case 1:} If $\mathcal{U}^i_k = \R^i_k$, then $\varphi_i(G,x) = \Phi_{\ell'}(G,x)$, and, since $(\hat \sigma^0, \hat \sigma^1) \Vdash_i \varphi_i(G^i, p_i)$, there exist indices $e_0 < \dots < e_{p_i} < \ell'$ such that $\Phi_{e_x}^{\hat \sigma^i \oplus Y}(e_x) \downarrow$ for every $x \leq p_i$. Since $(\hat \sigma^0, \hat \sigma^1) \Vdash_i \neg \varphi_i(G^i, p_i + 1)$, then, for every $y < \ell'$ not equal to one of the $e_x$, the condition $(\hat \sigma^0, \hat \sigma^1)$ forces $\Phi_{y}^{G \oplus Y}(y) \uparrow$. 
Hence, the condition $(\hat \sigma^0, \hat \sigma^1)$ forces $\R^i_k$ as $\ell' \geq k$.
\smallskip

\textbf{Case 2:} If $\mathcal{U}^i_k = \S^i_k$, then $\varphi_i(G,x) = \Psi_{\ell'}(G,x)$, and, since $(\hat \sigma^0, \hat \sigma^1) \Vdash_i \varphi_i(G^i, p_i)$, there exist indices $e_0 < \dots < e_{p_i} < \ell'$ such that $\card W_{e_x}^{\hat \sigma^i \oplus Y} \geq 2_{n+2}(\ell') = 2_n(\ell)$ for every $x \leq p_i$. Since $(\hat \sigma^0, \hat \sigma^1) \Vdash_i \neg \varphi_i(G^i, p_i + 1)$, then, for every $y < \ell'$ not equal to one of the $e_x$, the condition $(\hat \sigma^0, \hat \sigma^1)$ forces $\card W_{y}^{G^i \oplus Y} < 2_{n+2}(\ell')$.

Let $x \leq p_i$ and let $\rho$ be index of $T$ such that $\hat \sigma^i = \sigma_{\rho}^i$. Since the depth of $T$ is smaller than $2\ell'-1$, $\rho \in 2^{\leq2(\ell' - 1)}$ and, by definition of $\ell'$, $e(\ulcorner\varphi_0 \urcorner, \ulcorner \varphi_1 \urcorner, \sigma^0, \sigma^1, \rho, i, e_x) < 2^{2^{\ell'}} = \ell$.
By definition of the function $e$, $W_{e_x}^{\hat{\sigma}^i \oplus Y}$ equals $W^Y_{e(\langle \ulcorner\varphi_0\urcorner, \ulcorner\varphi_1\urcorner, \sigma^0, \sigma^1, \rho, i, e_x\rangle)}$, hence, 
$$\card W^Y_{e(\langle \ulcorner\varphi_0\urcorner, \ulcorner\varphi_1\urcorner, \sigma^0, \sigma^1, \rho ,i,e_x\rangle)} \geq 2_{n+2}(\ell') = 2_n(\ell)$$ 
and $(\hat \sigma^0, \hat \sigma^1) \Vdash_i W^Y_{e(\langle \ulcorner\varphi_0\urcorner, \ulcorner\varphi_1\urcorner, \sigma^0, \sigma^1, \rho, i, e_x\rangle)} \subseteq W_{e_x}^{G^i \oplus Y}$. Hence, $(\hat \sigma^0, \hat \sigma^1)$ forces $\mathcal{S}_k^i$. 
\end{proof}

\textbf{Construction} 
As in \cite[Proposition 3.17]{houerou2025conservation}, we shall construct two sets
$G^0$ and $G^1$, being approached by a decreasing sequence $(\sigma^0_s,\sigma^1_s)$ of conditions, as well as their jump ${G^0}'$ and ${G^1}'$ being approached by sequences $({\sigma^0_s}')$ and $({\sigma^1_s}')$. We will also approach each $B^i$ by sequences $(B^i_s)_{s \in M}$.

Let $\sigma_0^0 = \sigma_0^1 = {\sigma^0_0}' = {\sigma^1_0}' = B^0_0 = B^1_0 = \emptyset$. Assume that $\sigma_s^0, \sigma_s^1, {\sigma^0_s}', {\sigma^1_s}', B^0_s$ and $B^1_s$ are defined for some $s \in M$, and let 
$$
k_s  = 1+\max \{|\sigma_s^0|, |\sigma_s^1|, |{\sigma^0_s}'|, |{\sigma^1_s}'|, |B^0_s|, |B^1_s|\}
$$
We will satisfy one of the five requirements $\R_{k_s}^0 \vee \R_{k_s}^0$, $\S_{k_s}^0 \vee \R_{k_s}^1$, $\R_{k_s}^0 \vee \S_{k_s}^1$, $\S_{k_s}^0 \vee \S_{k_s}^1$ or $\mathcal{T}^0_{k_s} \wedge \mathcal{T}^1_{k_s}$, depending on the value of $s$ modulo $5$.   
There are two cases:
\smallskip

\textbf{Case 1: The requirement is of the form $\mathcal{U}^0_{k_s} \vee \mathcal{U}^1_{k_s}$, with $\mathcal{U}^i_{k_s} \in \{\R^i_{k_s}, \S^i_{k_s}\}$.} In that case, by \Cref{lem:sads-forcing-eff-jump}, there exists some $i < 2$ and some extension $(\hat \sigma^0, \hat \sigma^1) \leq (\sigma^0_s, \sigma^1_s)$ forcing $\mathcal{U}^i_{k_s}$, we then let $\sigma^0_{s+1} = \hat \sigma^0$ and $\sigma^1_{s+1} = \hat \sigma^1$. Then, if $\mathcal{U}^i_{k_s} = \R^i_{k_s}$, let ${\sigma^i_{s+1}}'$ be such that $(\hat \sigma^0, \hat \sigma^1) \Vdash (G^i \oplus Y)' = {\sigma^i_{s+1}}'$ and let ${\sigma^{1-i}_{s+1}}'$, $B^0_{s+1}$ and $B^1_{s+1}$ be unchanged. If $\mathcal{U}^i_{k_s} = \S^i_{k_s}$, let $B^i_{s+1}(x) = 0$ for every $x \in [|B^i_s|, \log_2(\log_2(\ell)))$ and $B^i_{s+1}(\log_2(\log_2(\ell))) = 1$, where $\ell$ is the bound obtained from the fact that $(\hat \sigma^0, \hat \sigma^1) \Vdash_i \S^i_{k_s}$ and let ${\sigma^{0}_{s+1}}'$, ${\sigma^{1}_{s+1}}'$ and $B^{1-i}_{s+1}$ be unchanged.
\smallskip

\textbf{Case 2: The requirement is of the form $\mathcal{T}^0_{k_s} \wedge \mathcal{T}^1_{k_s}$.} Since both $U^0$ and $U^1$ are $M$-unbounded, there exists some condition $(\tau^0, \tau^1) \leq (\sigma^0_s, \sigma^1_s)$ with $\card \tau^i > \card \sigma^i_s$ for every $i < 2$. Let $\sigma^0_{s+1} = \tau^0$ and $\sigma^1_{s+1} = \tau^1$ and leave ${\sigma^{0}_{s+1}}'$, ${\sigma^{1}_{s+1}}'$, $B^0_{s+1}$ and $B^1_{s+1}$ unchanged. \\

The formula $\phi(s)$ stating that this construction can be pursued for $s$ steps (There exists a sequence $(\sigma_0^0, \dots,B_0^1),\dots, (\sigma_s^0, \dots,B_s^1)$ such that...) is equivalent to a $\Sigma_2(Y)$ formula, thanks to the fact that the construction is $Y'$-computable. Since we are not necessarily in a model of $\ISig_2$, there is no guarantee that $\phi(s)$ holds for every $s \in M$, and in some cases, if will actually only hold on a proper cut $I$. Nevertheless, the sequence of value $(k_s)_{s \in I}$ obtained during the construction must be $M$-unbounded, otherwise the formula $\phi(s)$ would be equivalent to a $\Delta_2(Y)$ formula that would hold for every $s \in M$ by $\BSig_2$ (which is equivalent to $\IDelta_2$), contradicting the fact that the sequence is $M$-bounded. Since this sequence is not $M$-bounded, for every $k \in M$, the requirements $\R^0_k \vee \R^1_k$, $\R^0_k \vee \S^1_k$, $\S^0_k \vee \R^1_k$, $\S^0_k \vee \S^1_k$, and $\T^0_k \wedge \T^1_k$ are satisfied (to satisfy one of these requirements, it is sufficient to satisfy it for a value $k_s \geq k$). Thus, for every $k \in M$ either $(\R^0_k \wedge \S^0_k \wedge \T^0_k)$ or $(\R^1_k \wedge \S^1_k \wedge \T^1_k)$ is satisfied, and, there exists some side $i < 2$ such that the requirements $\R^i_k, \S^i_k$, and $\T^i_k$ are simultaneously satisfied for every $k \in M$.

Let $G = \bigcup_{s \in I} \sigma^i_s$ and $\hat{B} = \bigcup_{s \in I} B^i_s$. $G$ is $\L$-monotonous as the $\sigma^i_s$ are $\L$-monotonous and compatible, $G$ is $M$-unbounded as the requirements $\T^i_k$ are satisfied for every $k \in M$. By construction, we get that $(G \oplus Y)' = \bigcup_{s \in I} {\sigma^{i}_s}'$, hence $Y' \geq (G \oplus Y)'$ as the sequence $({\sigma^{i}_s}')_{i \in I}$ is $Y'$-computable. Finally, $A$ is $\hat{B}$-block $(G \oplus Y)$-immune with cost $n + 2$. Indeed, every $b \in \hat{B}$ was added because some requirement $\S^i_{k_s}$ was forced (which happened $M$-unboundedly many times, hence $B$ is $M$-unbounded), hence there exists some $\ell \in B$ such that $b = \log_2(\log_2(\ell))$ and for every $x < b$, if $\card W_x^{G \oplus Y} \geq 2_n(\ell) = 2_{n+2}(b)$ then $W_e^Y \subseteq W^{G \oplus Y}_x$ and $\card W_e^Y \geq 2_n(\ell)$ for some $e < \ell$. This implies that $W^Y_e \not \subseteq A$ and therefore $W^{G \oplus Y}_x \not \subseteq A$ by the $B$-block $Y$-immunity with cost $n$ of $A$.

\end{proof}

In the proof of \Cref{prop:sads-eff}, the construction of the sets $G$ and $\hat{B}$ does not depends on the set $A$. Therefore, for any other set $\tilde{A} \subseteq M$ that is $B$-block $Y$-immune with cost $n$, we also get that $\tilde{A}$ is $\hat{B}$-block $(G \oplus Y)$-immune with cost $n+2$.

\subsection{Combinatorics of~$\COH$}

In order to prove the counterpart of \Cref{prop:sads-eff} for $\COH$, let us focus on the combinatorics of~$\COH$.
As mentioned, there is a correspondence between $\COH$ and $\WKL$ for $\Delta^0_2$-functionals, which can be formalized over~$\RCA_0+ \BSig_2$. It is important to elaborate on this correspondence to get a better grasp on the specificities of the $\COH$ step.

Given a uniformly computable sequence of sets $\vec{R} = R_0, R_1, \dots$, the construction of an $\vec{R}$-cohesive set can be understood as the interleaving of two kinds of operations: the \emph{growing} operation which consists of adding new elements to the cohesive set, and the \emph{cohesiveness} operation which, given a set~$R_n$, decides whether the remainder of the construction should commit to adding only elements from~$R_n$ or from $\overline{R}_n$.
The possible sequences of cohesiveness decisions can be represented as the infinite paths through a $\Sigma_1(\vec{R})$ binary tree as follows:

\begin{definition}
    For every instance $\vec{R} = R_0, R_1, \dots$ of $\COH$ and every $\rho \in 2^{<M}$, we write $\vec{R}_\rho := \bigcap_{\rho(i) = 0} \overline{R}_i \cap \bigcap_{\rho(i) = 1} R_i$ and let $T(\vec{R})$ to be the $\Sigma_1(\vec{R})$-tree $\{ \rho \in 2^{<M} : \exists x (x > |\rho| \wedge x \in \vec{R}_\rho) \}$.
\end{definition}

Note that a node $\rho$ is extensible in $T(\vec{R})$ iff $\vec{R}_\rho$ is infinite. Any infinite $\vec{R}$-cohesive set~$C$ gives a $\Delta_2(C)$-definition of a path~$P$ through~$T(\vec{R})$ by letting~$x \in P$ iff $C \subseteq^* R_x$. Conversely, Jockusch and Stephan~\cite{jockusch1993cohesive} showed that every degree whose jump computes a path through~$T(\vec{R})$ computes an infinite $\vec{R}$-cohesive set.

It follows that there exists a natural notion of forcing for producing $\vec{R}$-cohesive sets, using a restriction of computable Mathias forcing to reservoirs which are boolean combinations of the sets of the instance. This notion of forcing is parameterized by a fixed path~$P \in [T(\vec{R})]$.

\begin{definition}
A \emph{condition} is a pair of binary strings $(\sigma, \rho)$ such that $\rho \prec P$.    
\end{definition}

One can think of a condition $(\sigma, \rho)$ as the computable Mathias condition $(\sigma, \vec{R}_\rho \setminus \{0, \dots, |\sigma| \})$. The notion of extension is defined accordingly: a condition $(\tau, \mu)$ extends $(\sigma, \rho)$ if $\sigma \preceq \tau$, $\rho \preceq \mu$, and $\tau \setminus \sigma \subseteq \vec{R}_\rho$.

By many means, the situation for~$\COH$ is much simpler than for~$\SADS$: the notion of forcing for~$\COH$ admits a first-jump control very similar to Cohen forcing or computable Mathias forcing. In particular, it is non-disjunctive.
There is however a new difficulty: the set of conditions is $P$-computable, and therefore the overall construction will be $P \oplus \emptyset'$-computable instead of $\emptyset'$-computable. One can easily prove that the use of~$P$ is necessary. Indeed, if a computable instance~$\vec{R}$ of~$\COH$ admits a solution of low degree~$\mathbf{d}$, then $\mathbf{d}'$ computes a path through $T(\vec{R})$, hence the tree admits a $\emptyset'$-computable path, and therefore there is a computable $\vec{R}$-cohesive set.

Working in a model~$\M = (M, S) \models \RCA_0 + \BSig_2 + \neg \ISig_2$ topped by a set~$Y$, the statement \qt{the construction can be pursued for~$s$ steps} will be $\Sigma_1(P \oplus Y')$ instead of~$\Sigma_2(Y)$. The corresponding cut~$I$ will be $\Sigma_1(P \oplus Y')$, and if~$P$ is arbitrary, the sequence of value $(k_s)_{s \in I}$ will not necessarily be $M$-unbounded, as this would yield a $\Delta_1(P \oplus Y')$-definition of~$I$, which is not a contradiction. One must therefore ensure that $\M[P \oplus Y'] \models \IDelta_1$ to make the construction work. Thankfully, this is the case if the path~$P$ is chosen with some care.

Let $\RCA_0^*$ be the theory $\Psf^-$ together with the $\Delta^0_1$-induction scheme, the $\Delta^0_1$-comprehension scheme, and the statement of the totality of the exponential ($\exp$). In particular, $\RCA_0^* \vdash \BSig_1$, and if $\M = (M, S)$ is a model of~$\RCA_0 + \BSig_2$, then $\M' = (M, \Definable{\Delta_2}{\M})$ is a model of~$\RCA_0^*$. The theory~$\RCA_0^*$ was introduced by Simpson and Smith~\cite{simpson1986factorization}, who proved the following theorem:

\begin{theorem}[Simpson–Smith~\cite{simpson1986factorization}]\label[theorem]{thm:wkl-rca0s}
Every countable model of~$\RCA_0^*$ can be $\omega$-extended into a model of~$\RCA_0^* + \WKL$.
\end{theorem}

This is exactly what we need to prove the existence of a path~$P$ such that $\M[P \oplus Y'] \models \IDelta_1$.

\begin{lemma}\label[lemma]{lem:bsig2-decision-path}
    Let $\M = (M,S) \models \RCA_0 + \BSig_2$ be a countable model topped by some set $Y$. For every instance $\vec{R}$ of $\COH$ in $S$, there exists some path $P \in [T(\vec{R})]$ such that $\M[P \oplus Y'] \models \RCA_0^*$.
\end{lemma}

\begin{proof}
Since $\M \models \RCA_0 + \BSig_2$, $\M[Y'] \models \RCA_0^*$.
The tree $T(\vec{R})$ is $\Delta_1(Y')$ and is therefore contained in $\M[Y']$. Furthermore, by $\Delta_1(Y')$-induction (which holds since $\M[Y'] \models \RCA_0^*$), one can prove that for every~$k \in M$, there is some $\rho$ of length~$k$ such that $\vec{R}_\rho$ is $M$-unbounded. It follows that $T(\vec{R})$ is $M$-unbounded and is therefore an instance of $\WKL$ in $\M[Y']$.
Since $\M[Y'] \models \RCA_0^*$, by \Cref{thm:wkl-rca0s}, there exists some path $P$ in $T(\vec{R})$ such that $\M[P \oplus Y'] \models \RCA_0^*$.
\end{proof}

\subsection{The $\COH$ step}

We now prove the existence of $\omega$-extensions which add solutions to instances of~$\COH$ while preserving the fact that the instance of~$\BRT^2_2$ has no solution.

\begin{proposition}\label[proposition]{prop:coh-eff}
    Let $\M = (M,S) \models \RCA_0 + \BSig_2$ be a countable model topped by some set $Y$. Let $B$ be an $M$-unbounded $\Delta_2(Y)$ set and fix $n \in \omega$. Then, for every set $A \subseteq M$ that is $B$-block $Y$-immune with cost $n$, every instance $\vec{R}$ of $\COH$ in $S$ and every path $P \in [T(\vec{R})]$ such that $\M[P \oplus Y'] \models \RCA_0^*$, there exists some $M$-unbounded set $G \subseteq M$ and some $M$-unbounded $P \oplus Y'$-computable set $\hat{B}$ such that:
    
    \begin{itemize}
        \item $G$ is $\vec{R}$-cohesive 
        \item $P \oplus Y' \equiv_T (G \oplus Y)'$ and in particular $\M[G] \models \RCA_0 + \BSig_2$
        \item $A$ is $\hat{B}$-block $(G \oplus Y)$-immune with cost $n+2$.
    \end{itemize}
\end{proposition}

\begin{proof}
We will define the set~$G$ using an infinite $\Delta_1(P \oplus Y')$ decreasing sequence of Mathias-like conditions. As in \Cref{prop:sads-eff}, we will assume that all the elements of $B$ are of the form $2^{2^k}$ for some $k \in M$.

\begin{definition}
A \emph{condition} is a pair of $M$-finite binary strings $(\sigma,\rho)$ such that $\rho \prec P$.
We have $(\sigma_2,\rho_2) \leq (\sigma_1,\rho_1)$ if $\sigma_2 \succeq \sigma_1$, $\rho_2 \succeq \rho_1$, and $\sigma_2 - \sigma_1 \subseteq \vec{R}_{\rho_1}$ (where $\sigma_2 - \sigma_1$ is the difference between the sets corresponding to the two sequences).
\end{definition}

As mentioned, one can think of a condition $(\sigma, \rho)$ as a Mathias condition $(\sigma, \vec{R}_\rho)$. Requiring that $\rho \prec P$ ensures that $\vec{R}_\rho$ is $\M$-infinite. When increasing the length of~$\rho$, we make progress in cohesiveness.
Note that being a condition is a $\Delta_1(P)$ predicate, as it simply requires checking that $\rho \prec P$, while the extension relation is $\Delta_1(Y)$. 

A sequence $\tau$ is said to be \emph{compatible} with a condition $(\sigma,\rho)$ if $\sigma \preceq \tau$ and  $\tau - \sigma \subseteq \vec{R}_{\rho}$.

\begin{definition}
Let~$(\sigma, \rho)$ be a condition, and $e,x \in M$, we write:
\begin{enumerate}
    \item[1.] $(\sigma, \rho) \Vdash \Phi^{G \oplus Y}_e(x)\downarrow$ if $\Phi^{\sigma \oplus Y}_e(x)[t]\downarrow$
    \item[2.] $(\sigma, \rho) \Vdash \Phi^{G \oplus Y}_e(x)\uparrow$ if for every~$\tau$ compatible with~$(\sigma, \rho)$, $\Phi^{\tau \oplus Y}_e(x)\uparrow$.
\end{enumerate}
\end{definition}

Note that the relations $(\sigma, \rho) \Vdash \Phi^{G \oplus Y}_e(x)\downarrow$
and $(\sigma, \rho) \Vdash \Phi^{G \oplus Y}_e(x)\uparrow$ are $\Delta_0(Y)$ and $\Pi_1(Y)$, respectively. \\

To ensure that $P \oplus Y' \geq_T (G \oplus Y)'$ and that $A$ is $\hat{B}$-block $(G \oplus Y)$-immune with cost $n+2$, we will also construct $(G \oplus Y)'$ and $\hat{B}$ in parallel. 

For this, we will need to satisfy three kind of requirements for every~$k \in M$:
\begin{itemize}
    \item $\mathcal{R}_k$: there exists some $\sigma' \in 2^k$ such that $(G \oplus Y)' \uh_k = \sigma'$.
    \item $\mathcal{S}_{k}$: there exists some $\ell > 2^{2^k}$ in $B$, such that, for every $x < \log_2(\log_2(\ell))$, if $\card W^{G \oplus Y}_x \geq 2_n(\ell)$, then there is some $e < \ell$ such that $\card W^Y_e \geq 2_n(\ell)$ and $W^Y_e \subseteq W_x^{G \oplus Y}$.
    \item $\mathcal{T}_k$: $\exists x > k, x \in G$ and $G \subseteq^* \vec{R}_{P \uh k}$
\end{itemize}

The notations $(\sigma, \rho) \Vdash \mathcal{R}_k$ and $(\sigma, \rho) \Vdash \mathcal{S}_k$ are defined similarly as in \Cref{prop:sads-eff}.

\begin{lemma}\label[lemma]{lem:coh-eff-forcing-jump}
    Let $(\sigma, \rho)$ be a condition. For every~$k \in M$, there exists some $M$-finite $\sigma' \in 2^k$ and an extension $(\tau, \rho) \leq (\sigma, \rho)$ such that $(\tau, \rho) \Vdash (G \oplus Y)' \uh_k = \sigma'$.
\end{lemma}

\begin{proof}
Same as \cite[Lemma 2.11]{houerou2025conservation}.
\end{proof}

\begin{lemma}\label[lemma]{lem:coh-eff-forcing-eff}
    For every condition $(\sigma, \rho)$ and every~$k \in M$, there exists some extension $(\tau, \rho) \leq (\sigma, \rho)$ such that $(\tau, \rho) \Vdash \mathcal{S}_{k}$.
\end{lemma}

\begin{proof}
For every finite set $F \subset M$, every $x \in F$ and every $t \in M$, let $W_{e(\langle F, x,t \rangle)}^Y$ search the smallest $\tau$ compatible with $(\sigma, \rho)$ such that $\card W^{\tau \oplus Y}_{x'} \geq 2_{n+2}(t)$ for every $x' \in F$ and output the elements of $W^{\tau \oplus Y}_{x}$. 

Let $S : M \to M$ be the function which send $t$ to the biggest value taken by $e(\langle F,x,t\rangle)$ for every $F \subseteq t$ and $x \in t$. Assuming a reasonable encoding of $e$, this function grows slower than $t \mapsto 2^{2^t}$, indeed, every $F \subseteq t$ and every $x \in t$ can be encoded using less than $\alpha t$ bits for some $\alpha \in \omega$ and the same holds for the corresponding value taken by $e$.

Hence, there exists some $\ell \in B$ such that $\ell > 2^{2^k}$ and such that $\ell' := \log_2(\log_2(\ell))$ satisfies $S(\ell') < 2^{2^{\ell'}} = \ell$.

Let $\mathcal{F} = \{ F \subseteq \ell' : (\forall x \in F) \card W_{e(\langle F, x, \ell' \rangle)}^Y \geq 2_{n+2}(\ell') \}$, the set $\mathcal{F}$ is $\Sigma_1(Y)$, non empty (contains $\emptyset$) and bounded, hence, by $\ISig_1$ we can pick some $F \in \mathcal{F}$ maximal for the inclusion. 

Since $F \in \mathcal{F}$, there exists some $\tau$ compatible with $(\sigma, \rho)$ such that $\card W^{\tau \oplus Y}_{x} \geq 2_{n+2}(\ell')  = 2_n(\ell)$ for every $x \in F$. The condition $(\tau, \rho)$ thus forces $\card W^{G \oplus Y}_{x} \geq 2_n(\ell)$ for every $x \in F$ and, by the maximality assumption, forces $\card W^{G \oplus Y}_{x} < 2_n(\ell)$ for every $x \in \{0, \dots, \ell'-1\} \setminus F$, as otherwise $F \cup \{x\}$ would be in $\mathcal{F}$. Moreover, for every $x \in F$, we have $\card W_{e(\langle F, x,\ell' \rangle)}^Y \geq 2_n(\ell)$ and $W_{e(\langle F, x,\ell' \rangle)}^Y = W^{\tau \oplus Y}_x \subseteq W^{G \oplus Y}_{x}$.

Thus, $(\tau, \rho) \Vdash \S_k$.
\end{proof}

\begin{lemma}\label[lemma]{lem:coh-eff-forcing-size}
    For every condition $(\sigma, \rho)$ and every~$k \in M$, there exists some extension $(\tau, \rho') \leq (\sigma, \rho)$ such that $(\tau, \rho') \Vdash \mathcal{T}_{k}$.
\end{lemma}

\begin{proof}
Since $\rho \prec P$, the set $\vec{R}_{\rho}$ is $\M$-infinite and there exists some $\tau$ compatible with $(\sigma, \rho)$ such that $|\tau| \geq k$, making $(\tau, \rho) \Vdash \exists x \in G, x \geq k$. Then taking $\rho' = P \uh k$ ensures that every $G$ in the cone generated by $(\tau, \rho')$ will be such that $G \setminus \tau \subseteq \vec{R}_{P \uh k}$, hence that $(\tau, \rho') \Vdash \T_k$.
\end{proof}

\textbf{Construction:} The construction will build the set $G$ using a $P \oplus Y'$-computable decreasing sequence $(\sigma_s, \rho_s)$ of conditions. In parallel, the sets $(G \oplus Y)'$ and $\hat{B}$ will be approached by sequences $(\sigma_s')$ and $(B_s)$ of finite chains.

Let $\sigma_0 = \rho_0 = \sigma_0' = B_0 = \emptyset$. Assume that $\sigma_s, \rho_s, \sigma_s'$ and $B_s$ have been defined for some $s \in M$ and let $$k_s = 1 + \max \{|\sigma_s|, |\rho_s|, |\sigma_s'|, |B_s|\}$$

Using \Cref{lem:coh-eff-forcing-jump}, \Cref{lem:coh-eff-forcing-eff} and \Cref{lem:coh-eff-forcing-size}, there exists some extension $(\tau, \rho') \leq (\sigma_s, \rho_s)$ forcing $\R_{k_s}, \S_{k_s}$, and $\T_{k_s}$. Furthermore, such an extension can be found $(P \oplus Y')$-effectively. We then let $\sigma_{s+1} = \tau$, $\rho_{s+1} = \rho'$, $\sigma'_{s+1} = \sigma'$ where $\sigma'$ is the sequence obtained from the fact that $\R_{k_s}$ is forced and let $B_{s+1}(x) = 0$ for every $x \in [|B_s|, \log_2(\log_2(\ell)))$ and $B_{s+1}(\log_2(\log_2(\ell))) = 1$ where $\ell$ is the bound obtained from the fact that $\S_{k_s}$ is forced. We then proceed to the next stage of the construction. \\

The formula $\phi(s)$ stating that this construction can be pursued for $s$ steps (There exists a sequence $(\sigma_0, \rho_0, \sigma'_0,B_0),\dots, (\sigma_s, \rho_s, \sigma_s',B_s)$ such that...) is equivalent to a $\Sigma_1(P \oplus Y')$ formula, thanks to the fact that the construction is $P \oplus Y'$-computable. Since we are not necessarily in a model of $\ISig_1(P \oplus Y')$, there is no guarantee that $\phi(s)$ holds for every $s \in M$, and in some cases, if will actually only hold on a proper cut $I$. Nevertheless, the sequence of value $(k_s)_{s \in I}$ obtained during the construction must be $M$-unbounded, otherwise the formula $\phi(s)$ would be equivalent to a $\Delta_1(P \oplus Y')$ formula that would hold for every $s \in M$ as $\M[P \oplus Y'] \models \IDelta_1(P \oplus Y')$ , contradicting the fact that the sequence is $M$-bounded. Since this sequence is not $M$-bounded, for every $k \in M$, the requirements $\R_k$, $\S_k$, and $\T_k$ are satisfied (to satisfy one of these requirements, it is sufficient to satisfy it for a value $k_s \geq k$).

Thus, this construction defines two sets $G = \bigcup_{s \in I} \sigma_s$ and $\hat{B} = \bigcup_{s \in I} B_s$ that are $(P \oplus Y')$-computable and such that $(G \oplus Y)'$ is also $(P \oplus Y')$-computable. We also get that $(G \oplus Y)' \geq_T P$, indeed, using $(G \oplus Y)'$, we can check if $G \subseteq^* R_i$ or if $G \subseteq^* \bar{R_i}$ for every $i \in M$ and thus find the value of $P(i)$, hence $(G \oplus Y)' \equiv_T P \oplus Y'$ and $\hat{B}$ is $\Delta_2(G \oplus Y)$. Since the requirements $\T_k$ are satisfied, the set $G$ is $M$-unbounded and $\vec{R}$-cohesive. Finally, for the same reasons as in \Cref{prop:sads-eff}, we get that $A$ is $\hat{B}$-block $(G \oplus Y)$-immune with cost $n+2$.
\end{proof}

As in \Cref{prop:sads-eff}, the previous construction makes no reference of the set $A$. Thus, for every set $\tilde{A} \subseteq M$ that is $B$-block $Y$-immune with cost $n$ will also be $\hat{B}$-block $(G \oplus Y)$-immune with cost $n+2$.

\begin{corollary}\label[corollary]{cor:coh-eff-complete}
Let $\M = (M,S) \models \RCA_0 + \BSig_2$ be a countable model topped by some set $Y$. Let $B$ be an $M$-unbounded $\Delta_2(Y)$ set and fix $n \in \omega$. Then, for every set $A \subseteq M$ that is $B$-block $Y$-immune with cost $n$ and every instance $\vec{R}$ of $\COH$ in $S$, there exists some $M$-unbounded set $G \subseteq M$ and some $M$-unbounded $(G \oplus Y)'$-computable set $\hat{B}$ such that:
    \begin{itemize}
        \item $G$ is $\vec{R}$-cohesive 
        \item $\M[G] \models \RCA_0 + \BSig_2$
        \item $A$ is $\hat{B}$-block $(G \oplus Y)$-immune with cost $n+2$.
    \end{itemize}
\end{corollary}
\begin{proof}
Immediate by \Cref{prop:coh-eff} and \Cref{lem:bsig2-decision-path}.
\end{proof}

\subsection{Main result}

We are now ready to prove \Cref{main:bsig2-not-isig2-ads-not-urt22}.

\begin{repmaintheorem}{main:bsig2-not-isig2-ads-not-urt22}
Every countable topped model of $\RCA_0 + \BSig_2 + \neg \ISig_2$ can be $\omega$-extended into a model of $\RCA_0 + \ADS + \neg \BRT^2_2$.    
\end{repmaintheorem}
\begin{proof}
Let $\M_0 = (M, S_0)$ be a countable model of $\RCA_0 + \BSig_2 + \neg \ISig_2$, topped by a set~$Z_0 \in S_0$. Let $I \subseteq M$ be a $\Sigma_2(Z_0)$-definable proper cut, witnessing that $\M_0 \models \neg \ISig_2$, and let $B_0 \subseteq M$ be a $\Delta_2(\M_0)$-definable $M$-unbounded set with cardinality~$I$.
By the proof of \Cref{main:bounded-non-comp-non-isigma2}, $\M_0$ contains a stable instance~$f : [M]^2 \to 2$ of $\BRT^2_2$ with no $\Delta_1(Z_0)$-definable solution, and such that the set $A = \{ x \in M : \forall^\infty y f(x, y) = 0 \}$ is $B_0$-block $Z_0$-immune with cost~2.

Since~$M$ is countable, there is a (non necessarily computable) reordering $(\Gamma_n)_{n \in \omega}$ of the sequence $(\Phi_x)_{x \in M}$ of all Turing functionals with codes in~$M$.
By \Cref{prop:sads-eff,cor:coh-eff-complete}, define a sequence $(\M_n)_{n \in \omega}$ of countable models of $\RCA_0 + \BSig_2$ such that for every~$n \in \omega$,
\begin{itemize}
    \item[(1)] $\M_{n+1}$ is topped by some set~$Z_{n+1}$ and $\omega$-extends $\M_n$;
    \item[(2)] $A$ is $B_n$-block $Z_n$-immune with some cost~$c_n \in \omega$;
    \item[(3)] if $n = \langle 0, a, b\rangle$ 
    and $\Gamma_a^{Z_b}$ is an instance $\L = (M, <_\L)$ of~$\SADS$, then $\M_{n+1}$ contains an $M$-unbounded $\L$-ascending or $\L$-descending sequence;
    \item[(4)] if $n = \langle 1, a, b\rangle$ and $\Gamma_a^{Z_b}$ is an instance $\vec{R}$ of~$\COH$, then $\M_{n+1}$ contains an $M$-unbounded $\vec{R}$-cohesive set.
\end{itemize}
Let $\Nc = \bigcup_{n \in \omega} \M_n$. Then $\Nc \models \RCA_0 + \BSig_2$.
By (1), $\Nc$ $\omega$-extends $\M_0$. By (2) $A$ is immune relative to every set in~$\Nc$, so $f$ is an instance of $\BRT^2_2$ with no solution in~$\Nc$. It follows that $\Nc \not \models \BRT^2_2$.
By (3) and (4), $\Nc \models \SADS + \COH$, so by Hirschfeldt and Shore~\cite{hirschfeldt_combinatorial_2007}, $\Nc \models \ADS$.
\end{proof}

\begin{corollary}
$\RCA_0 + \ADS \not \vdash \BRT_2^2$
\end{corollary}
\begin{proof}
Let $M$ be a countable model of $\mathsf{I}\Sigma_1  + \mathsf{B}\Sigma_2 +  \neg \mathsf{I}\Sigma_2$.
Let $\M = (M, S)$ be the model whose second-order part $S$ consists of the $\Delta_1$-definable sets with parameters in~$M$. By Friedman~\cite{friedman1976systems}, $\M \models \RCA_0 + \BSig_2 + \neg \ISig_2$. Moreover, $\M$ is countable, and topped by~$\emptyset$, so by \Cref{main:bsig2-not-isig2-ads-not-urt22}, there is an $\omega$-extension~$\Nc$ of~$\M$ such that $\Nc \models \RCA_0 + \ADS + \neg \BRT^2_2$. Thus, $\RCA_0 + \ADS \not \vdash \BRT_2^2$.
\end{proof}

\begin{corollary}
$\RCA_0 + \ADS + \neg \BRT^2_2$ is $\Pi^1_1$-conservative over $\RCA_0 + \BSig_2 + \neg \ISig_2$.
\end{corollary}
\begin{proof}
Let $\phi = \forall X \psi(X)$ be a $\Pi^1_1$-sentence such that $\RCA_0 + \BSig_2 + \neg \ISig_2 \not \vdash \phi$.
By the completeness theorem and the downward Löwenheim-Skolem theorem, there exists a countable model $\M = (M, S) \models \RCA_0 + \BSig_2 + \neg \ISig_2 + \neg \phi$. Let $X \in S$ be such that $\M \models \neg \psi(X)$ and $Y \in S$ be such that $\ISig_2(Y)$ fails. Let $\M_1 =  (M, \Definable{\Delta_1}{X, Y})$. By Friedman~\cite{friedman1976systems}, $\M_1 \models \RCA_0 + \BSig_2$. Moreover, since $\psi$ is arithmetic, $\M_1 \models \neg \psi(X)$, and since $Y \in \M_1$, $\M_1 \models \neg \ISig_2$. Last, $\M_1$ is topped by $X \oplus Y$, so one can apply \Cref{main:bsig2-not-isig2-ads-not-urt22} to get an $\omega$-extension $\Nc$ of~$\M_1$ such that $\Nc \models \RCA_0 + \BSig_2 + \ADS + \neg \BRT^2_2$. In particular, $X \in \Nc$ and $\Nc \models \neg \psi(X)$, so $\RCA_0 + \BSig_2 + \ADS + \neg \BRT^2_2 \not \vdash \phi$.
\end{proof}

\section{Open questions}

We conclude this article with a few remaining open questions. Despite being provable over a sufficient amount of induction, $\BRT^2_2$ is not a $\Pi^1_1$-sentence, hence the fact that $\RCA_0 + \BSig_2 \not \vdash \BRT^2_2$ does not rule out the possibility that $\BRT^2_2$ has the same $\Pi^1_1$-consequences as $\BSig_2$.

\begin{question}\label[question]{quest:urt22-pi11-bsig2}
Is $\RCA_0 + \BRT^2_2$ a $\Pi^1_1$-conservative extension of~$\RCA_0 + \BSig_2$?
\end{question}

Since $\BRT^2_2$ follows from $\RT^2_2$, a negative answer to \Cref{quest:urt22-pi11-bsig2} would prove that $\RCA_0 + \RT^2_2$ is not a $\Pi^1_1$-conservative extension of~$\RCA_0 + \BSig_2$, which is a major open question in reverse mathematics.

Belanger (personal communication) claimed that $\RCA_0 + \WKL + \BSig_2 \vdash \BRT^2_2$, but the proof was lost. We therefore leave the question open. 

\begin{question}\label[question]{quest:wkl-bsig2-brt22}
Does $\RCA_0 + \WKL + \BSig_2 \vdash \BRT^2_2$?
\end{question}

Last, the Chain AntiChain principle ($\CAC$) has very similar features to~$\ADS$. It admits both an asymmetric and symmetric forcing construction, with a notion of split pairs (\cite[Theorem 14]{patey_partial_2018}). Hirschfeldt and Shore~\cite[Corollary 3.11]{hirschfeldt_combinatorial_2007} also proved that $\CAC$ does not imply the existence of a DNC function over~$\RCA_0$.

\begin{question}\label[question]{quest:cac-brt22}
Does $\RCA_0 + \CAC \vdash \BRT^2_2$?
\end{question}

The difficulty of adapting the proof of \Cref{prop:sads-eff} to a stable partial order comes the complexity of finding an extension forcing either $\neg \varphi_0(G^0)$ or $\neg \varphi_1(G^1)$ in the case of the absence of a split pair $(\sigma^0, \sigma^1)$ such that $\varphi_0(\sigma^0)$ and $\varphi_1(\sigma^1)$ both hold.

Note that any positive answer to either \Cref{quest:wkl-bsig2-brt22} or \Cref{quest:cac-brt22} would also give a positive answer to \Cref{quest:urt22-pi11-bsig2}.

\section*{Acknowledgement}

The authors are thankful to Emanuele Frittaion for drawing our attention to the bounded Ramsey's theorem project and its questions. We are also thankful to David Belanger, Paul Shafer, and Keita Yokoyama for interesting discussions.
The authors are thankful to the anonymous referee for his improvement suggestions.

\bibliographystyle{plain}
\bibliography{biblio}

@book {hajek1998metamathematics,
    AUTHOR = {H\'{a}jek, Petr and Pudl\'{a}k, Pavel},
     TITLE = {Metamathematics of first-order arithmetic},
    SERIES = {Perspectives in Mathematical Logic},
      NOTE = {Second printing},
 PUBLISHER = {Springer-Verlag, Berlin},
      YEAR = {1998},
     PAGES = {xiv+460},
      ISBN = {3-540-63648-X},
   MRCLASS = {03-02 (03D15 03F30 03H15 11U09 11U10 68Q15)},
  MRNUMBER = {1748522},
}

@article {chong2021pi11,
    AUTHOR = {Chong, C. T. and Slaman, Theodore A. and Yang, Yue},
     TITLE = {{$\Pi^1_1$}-conservation of combinatorial principles weaker
              than {R}amsey's theorem for pairs},
   JOURNAL = {Adv. Math.},
  FJOURNAL = {Advances in Mathematics},
    VOLUME = {230},
      YEAR = {2012},
    NUMBER = {3},
     PAGES = {1060--1077},
      ISSN = {0001-8708},
   MRCLASS = {03B30 (03D80 03F30 03F35 05D10)},
  MRNUMBER = {2921172},
MRREVIEWER = {Denis R. Hirschfeldt},
       DOI = {10.1016/j.aim.2012.02.025},
       URL = {https://doi.org/10.1016/j.aim.2012.02.025},
}

@article{fiori2021isomorphism,
  title={An isomorphism theorem for models of {W}eak {K}\"onig's {L}emma without primitive recursion},
  author={Fiori-Carones, Marta and Ko{\l}odziejczyk, Leszek Aleksander and Wong, Tin Lok and Yokoyama, Keita},
  journal={Journal of the European Mathematical Society},
  note = {To appear},
  year={2021}
}

@article{bovykin2017strength,
  title={The strength of infinitary Ramseyan principles can be accessed by their densities},
  author={Bovykin, Andrey and Weiermann, Andreas},
  journal={Annals of Pure and Applied Logic},
  volume={168},
  number={9},
  pages={1700--1709},
  year={2017},
  publisher={Elsevier}
}

@book {hirst1987combinatorics,
    AUTHOR = {Hirst, Jeffry Lynn},
     TITLE = {C{OMBINATORICS} {IN} {SUBSYSTEMS} {OF} {SECOND} {ORDER}
              {ARITHMETIC}},
      NOTE = {Thesis (Ph.D.)--The Pennsylvania State University},
 PUBLISHER = {ProQuest LLC, Ann Arbor, MI},
      YEAR = {1987},
     PAGES = {153},
   MRCLASS = {Thesis},
  MRNUMBER = {2635978},
       URL =
              {http://gateway.proquest.com/openurl?url_ver=Z39.88-2004&rft_val_fmt=info:ofi/fmt:kev:mtx:dissertation&res_dat=xri:pqdiss&rft_dat=xri:pqdiss:8728018},
}

@article {chong2017inductive,
    AUTHOR = {Chong, C. T. and Slaman, Theodore A. and Yang, Yue},
     TITLE = {The inductive strength of {R}amsey's {T}heorem for {P}airs},
   JOURNAL = {Adv. Math.},
  FJOURNAL = {Advances in Mathematics},
    VOLUME = {308},
      YEAR = {2017},
     PAGES = {121--141},
      ISSN = {0001-8708},
   MRCLASS = {03B30 (03D80 03F30 03F35 03H15 05D10)},
  MRNUMBER = {3600057},
MRREVIEWER = {Denis R. Hirschfeldt},
       DOI = {10.1016/j.aim.2016.11.036},
       URL = {https://doi.org/10.1016/j.aim.2016.11.036},
}

@incollection {seetapun1995strength,
    AUTHOR = {Seetapun, David and Slaman, Theodore A.},
     TITLE = {On the strength of {R}amsey's theorem},
      NOTE = {Special Issue: Models of arithmetic},
   JOURNAL = {Notre Dame J. Formal Logic},
  FJOURNAL = {Notre Dame Journal of Formal Logic},
    VOLUME = {36},
      YEAR = {1995},
    NUMBER = {4},
     PAGES = {570--582},
      ISSN = {0029-4527},
   MRCLASS = {03F35 (03C62)},
  MRNUMBER = {1368468},
MRREVIEWER = {Roman Murawski},
       DOI = {10.1305/ndjfl/1040136917},
       URL = {https://doi.org/10.1305/ndjfl/1040136917},
}

@article{liu2012rt22,
  title={{RT}22 does not imply {WKL}0},
  author={Liu, Jiayi},
  journal={The Journal of Symbolic Logic},
  volume={77},
  number={2},
  pages={609--620},
  year={2012},
  publisher={Cambridge University Press}
}

@article {kreuzer2012primitive,
    AUTHOR = {Kreuzer, Alexander P.},
     TITLE = {Primitive recursion and the chain antichain principle},
   JOURNAL = {Notre Dame J. Form. Log.},
  FJOURNAL = {Notre Dame Journal of Formal Logic},
    VOLUME = {53},
      YEAR = {2012},
    NUMBER = {2},
     PAGES = {245--265},
      ISSN = {0029-4527},
   MRCLASS = {03F35 (03B30)},
  MRNUMBER = {2925280},
MRREVIEWER = {Mariko Yasugi},
       DOI = {10.1215/00294527-1715716},
       URL = {https://doi.org/10.1215/00294527-1715716},
}

@article{friedman1976systems,
  title={Systems on second order arithmetic with restricted induction i, ii},
  author={Friedman, Harvey M},
  journal={J. Symb. Logic},
  volume={41},
  pages={557--559},
  year={1976}
}

@article {houerou2025conservation,
    AUTHOR = {Le Hou\'{e}rou, Quentin and Patey, Ludovic Levy and Yokoyama,
              Keita},
     TITLE = {Conservation of {R}amsey's theorem for pairs and
              well-foundedness},
   JOURNAL = {Trans. Amer. Math. Soc.},
  FJOURNAL = {Transactions of the American Mathematical Society},
    VOLUME = {378},
      YEAR = {2025},
    NUMBER = {3},
     PAGES = {2157--2186},
      ISSN = {0002-9947},
   MRCLASS = {03B30 (03F30 05D10)},
  MRNUMBER = {4866362},
       DOI = {10.1090/tran/9336},
       URL = {https://doi.org/10.1090/tran/9336},
}

@incollection {simpson1986factorization,
    AUTHOR = {Simpson, Stephen G. and Smith, Rick L.},
     TITLE = {Factorization of polynomials and {$\Sigma^0_1$} induction},
      NOTE = {Special issue: second Southeast Asian logic conference
              (Bangkok, 1984)},
   JOURNAL = {Ann. Pure Appl. Logic},
  FJOURNAL = {Annals of Pure and Applied Logic},
    VOLUME = {31},
      YEAR = {1986},
    NUMBER = {2-3},
     PAGES = {289--306},
      ISSN = {0168-0072},
   MRCLASS = {03F35 (12E05)},
  MRNUMBER = {854297},
MRREVIEWER = {M. Dubiel},
       DOI = {10.1016/0168-0072(86)90074-6},
       URL = {https://doi.org/10.1016/0168-0072(86)90074-6},
}

@incollection{specker_ramseys_1971,
    title = {Ramsey's theorem does not hold in recursive set theory},
    volume = {61},
    booktitle = {Studies in {Logic} and the {Foundations} of {Mathematics}},
    publisher = {Elsevier},
    author = {Specker, Ernst},
    year = {1971},
    pages = {439--442},
}

@article{hirschfeldt_combinatorial_2007,
    title = {Combinatorial principles weaker than {Ramsey}'s theorem for pairs},
    volume = {72},
    number = {1},
    journal = {Journal of Symbolic Logic},
    author = {Hirschfeldt, Denis R. and Shore, Richard A.},
    year = {2007},
    pages = {171--206},
}

@article{manuel_lerman_separating_2013,
    title = {Separating principles below {Ramsey}'s {Theorem} for {Pairs}},
    author = {Manuel Lerman, Reed Solomon, Henry Towsner},
    year = {2013},
    note = {preprint},
}

@incollection{paris_n-collection_1978,
    title = {{$\Sigma_n$}-{Collection} {Schemas} in {Arithmetic}},
    volume = {96},
    isbn = {978-0-444-85178-9},
    url = {https://linkinghub.elsevier.com/retrieve/pii/S0049237X08720032},
    language = {en},
    urldate = {2021-04-09},
    booktitle = {Studies in {Logic} and the {Foundations} of {Mathematics}},
    publisher = {Elsevier},
    author = {Paris, J.B. and Kirby, L.A.S.},
    year = {1978},
    doi = {10.1016/S0049-237X(08)72003-2},
    pages = {199--209},
}

@book{simpson_subsystems_2009,
    title = {Subsystems of {Second} {Order} {Arithmetic}},
    publisher = {Cambridge University Press},
    author = {Simpson, Stephen G.},
    year = {2009},
}

@misc{simpson2021fragments,
Author = {Stephen G. Simpson and Keita Yokoyama},
Title = {Very weak fragments of weak Kőnig's lemma},
Year = {2021},
Eprint = {arXiv:2101.00636},
}

@article{corduan_reverse_2010,
    title = {Reverse mathematics and {Ramsey}'s property for trees},
    volume = {75},
    number = {03},
    journal = {Journal of Symbolic Logic},
    author = {Corduan, Jared and Groszek, Marcia J. and Mileti, Joseph R.},
    year = {2010},
    note = {Publisher: Cambridge Univ Press},
    pages = {945--954},
}

@article{chong_strength_2020,
    title = {On the strength of {Ramsey}'s theorem for trees},
    volume = {369},
    issn = {0001-8708},
    url = {https://doi.org/10.1016/j.aim.2020.107180},
    doi = {10.1016/j.aim.2020.107180},
    journal = {Advances in Mathematics},
    author = {Chong, C. T. and Li, Wei and Wang, Wei and Yang, Yue},
    year = {2020},
    mrnumber = {4093609},
    pages = {107180, 39},
}

@article{belanger_conservation_2015,
    title = {Conservation theorems for the {Cohesiveness} {Principle}},
    author = {Belanger, David R},
    year = {2015},
}

@article{patey_partial_2018,
    title = {Partial orders and immunity in reverse mathematics},
    volume = {7},
    issn = {2211-3568},
    doi = {10.3233/COM-170071},
    language = {English},
    number = {4},
    journal = {Computability},
    author = {Patey, Ludovic},
    year = {2018},
    pages = {323--339},
}

@article{jockusch1993cohesive,
  title={A cohesive set which is not high},
  author={Jockusch, Carl and Stephan, Frank},
  journal={Mathematical Logic Quarterly},
  volume={39},
  number={1},
  pages={515--530},
  year={1993},
  publisher={Wiley Online Library}
}

@incollection {jockusch1989fpf,
    AUTHOR = {Jockusch, Jr., Carl G.},
     TITLE = {Degrees of functions with no fixed points},
 BOOKTITLE = {Logic, methodology and philosophy of science, {VIII}
              ({M}oscow, 1987)},
    SERIES = {Stud. Logic Found. Math.},
    VOLUME = {126},
     PAGES = {191--201},
 PUBLISHER = {North-Holland, Amsterdam},
      YEAR = {1989},
   MRCLASS = {03D30},
  MRNUMBER = {1034562},
MRREVIEWER = {Barry Cooper},
       DOI = {10.1016/S0049-237X(08)70045-4},
       URL = {https://doi.org/10.1016/S0049-237X(08)70045-4},
}

@article{arslanov1977completeness,
  title={Completeness criteria for recursively enumerable sets and some general theorems on fixed points},
  author={Arslanov, MM and Nadirov, RF and Solov’ev, VD},
  journal={Soviet Mathematics (Izvestiya VUZ. Matematika)},
  volume={21},
  number={4},
  pages={1--4},
  year={1977}
}

@article {schmerl2000graph,
    AUTHOR = {Schmerl, James H.},
     TITLE = {Graph coloring and reverse mathematics},
   JOURNAL = {MLQ Math. Log. Q.},
  FJOURNAL = {MLQ. Mathematical Logic Quarterly},
    VOLUME = {46},
      YEAR = {2000},
    NUMBER = {4},
     PAGES = {543--548},
      ISSN = {0942-5616},
   MRCLASS = {03F35 (05C15)},
  MRNUMBER = {1791549},
MRREVIEWER = {Kazuyuki Tanaka},
       DOI = {10.1002/1521-3870(200010)46:4<543::AID-MALQ543>3.0.CO;2-E},
       URL =
              {https://doi.org/10.1002/1521-3870(200010)46:4<543::AID-MALQ543>3.0.CO;2-E},
}

@article {kierstead1981effective,
    AUTHOR = {Kierstead, Henry A.},
     TITLE = {An effective version of {D}ilworth's theorem},
   JOURNAL = {Trans. Amer. Math. Soc.},
  FJOURNAL = {Transactions of the American Mathematical Society},
    VOLUME = {268},
      YEAR = {1981},
    NUMBER = {1},
     PAGES = {63--77},
      ISSN = {0002-9947},
   MRCLASS = {03D45 (05A05 06A10)},
  MRNUMBER = {628446},
MRREVIEWER = {James H. Schmerl},
       DOI = {10.2307/1998337},
       URL = {https://doi.org/10.2307/1998337},
}

@article {gasarch1998reverse,
    AUTHOR = {Gasarch, William and Hirst, Jeffry L.},
     TITLE = {Reverse mathematics and recursive graph theory},
   JOURNAL = {MLQ Math. Log. Q.},
  FJOURNAL = {MLQ. Mathematical Logic Quarterly},
    VOLUME = {44},
      YEAR = {1998},
    NUMBER = {4},
     PAGES = {465--473},
      ISSN = {0942-5616},
   MRCLASS = {03F35 (03D45 05C15 05C45)},
  MRNUMBER = {1654281},
MRREVIEWER = {Xiaokang Yu},
       DOI = {10.1002/malq.19980440405},
       URL = {https://doi.org/10.1002/malq.19980440405},
}

@incollection {felsner1997online,
    AUTHOR = {Felsner, Stefan},
     TITLE = {On-line chain partitions of orders},
      NOTE = {Orders, algorithms and applications (Lyon, 1994)},
   JOURNAL = {Theoret. Comput. Sci.},
  FJOURNAL = {Theoretical Computer Science},
    VOLUME = {175},
      YEAR = {1997},
    NUMBER = {2},
     PAGES = {283--292},
      ISSN = {0304-3975},
   MRCLASS = {06A06 (68R05)},
  MRNUMBER = {1451634},
MRREVIEWER = {Volkmar Welker},
       DOI = {10.1016/S0304-3975(96)00204-6},
       URL = {https://doi.org/10.1016/S0304-3975(96)00204-6},
}

@article {mirsky1971dual,
    AUTHOR = {Mirsky, L.},
     TITLE = {A dual of {D}ilworth's decomposition theorem},
   JOURNAL = {Amer. Math. Monthly},
  FJOURNAL = {American Mathematical Monthly},
    VOLUME = {78},
      YEAR = {1971},
     PAGES = {876--877},
      ISSN = {0002-9890},
   MRCLASS = {06.20},
  MRNUMBER = {288054},
MRREVIEWER = {Richard Stanley},
       DOI = {10.2307/2316481},
       URL = {https://doi.org/10.2307/2316481},
}

@misc{mimouni2025ramseylike,
Author = {Ahmed Mimouni and Ludovic Patey},
Title = {Ramsey-like theorems and immunities},
Year = {2025},
Eprint = {arXiv:2508.15597},
}

@article{jockusch_ramseys_1972,
    title = {Ramsey's theorem and recursion theory},
    volume = {37},
    number = {2},
    journal = {The Journal of Symbolic Logic},
    author = {Jockusch, Carl G},
    year = {1972},
    pages = {268--280},
}

@phdthesis{solda2021calibrating,
    type = {phd},
    title = {Calibrating the complexity of combinatorics: reverse mathematics and {Weihrauch} degrees of some principles related to {Ramsey}’s theorem},
    school = {University of Leeds},
    author = {Solda, Giovanni},
    year = {2021},
}

@article {rival1980adjacency,
    AUTHOR = {Rival, Ivan and Sands, Bill},
     TITLE = {On the adjacency of vertices to the vertices of an infinite
              subgraph},
   JOURNAL = {J. London Math. Soc. (2)},
  FJOURNAL = {Journal of the London Mathematical Society. Second Series},
    VOLUME = {21},
      YEAR = {1980},
    NUMBER = {3},
     PAGES = {393--400},
      ISSN = {0024-6107},
   MRCLASS = {05C35 (06A10)},
  MRNUMBER = {577715},
MRREVIEWER = {M. E. Watkins},
       DOI = {10.1112/jlms/s2-21.3.393},
       URL = {https://doi.org/10.1112/jlms/s2-21.3.393},
}

@article {fiori2024extraordinary,
    AUTHOR = {Fiori-Carones, Marta and Marcone, Alberto and Shafer, Paul and
              Sold\`a, Giovanni},
     TITLE = {({E}xtra)ordinary equivalences with the ascending/descending
              sequence principle},
   JOURNAL = {J. Symb. Log.},
  FJOURNAL = {The Journal of Symbolic Logic},
    VOLUME = {89},
      YEAR = {2024},
    NUMBER = {1},
     PAGES = {262--307},
      ISSN = {0022-4812},
   MRCLASS = {03B30 (03F35 05D10 06A06)},
  MRNUMBER = {4725676},
MRREVIEWER = {Steffen Lempp},
       DOI = {10.1017/jsl.2022.92},
       URL = {https://doi.org/10.1017/jsl.2022.92},
}

@article {bosek2012online,
    AUTHOR = {Bosek, Bart{\l}omiej and Felsner, Stefan and Kloch, Kamil and
              Krawczyk, Tomasz and Matecki, Grzegorz and Micek, Piotr},
     TITLE = {On-line chain partitions of orders: a survey},
   JOURNAL = {Order},
  FJOURNAL = {Order. A Journal on the Theory of Ordered Sets and its
              Applications},
    VOLUME = {29},
      YEAR = {2012},
    NUMBER = {1},
     PAGES = {49--73},
      ISSN = {0167-8094},
   MRCLASS = {91A46 (05C57 06A07)},
  MRNUMBER = {2948748},
       DOI = {10.1007/s11083-011-9197-1},
       URL = {https://doi.org/10.1007/s11083-011-9197-1},
}

@article {bosek2018subexponential,
    AUTHOR = {Bosek, Bart{\l}omiej and Kierstead, H. A. and Krawczyk, Tomasz
              and Matecki, Grzegorz and Smith, Matthew E.},
     TITLE = {An easy subexponential bound for online chain partitioning},
   JOURNAL = {Electron. J. Combin.},
  FJOURNAL = {Electronic Journal of Combinatorics},
    VOLUME = {25},
      YEAR = {2018},
    NUMBER = {2},
     PAGES = {Paper No. 2.28, 23},
   MRCLASS = {68R05 (05C15 06A07 68W27)},
  MRNUMBER = {3814262},
MRREVIEWER = {Yong Zhang},
       DOI = {10.37236/7231},
       URL = {https://doi.org/10.37236/7231},
}

@incollection {bosek2010subexponential,
    AUTHOR = {Bosek, Bart{\l}omiej and Krawczyk, Tomasz},
     TITLE = {The sub-exponential upper bound for on-line chain
              partitioning},
 BOOKTITLE = {2010 {IEEE} 51st {A}nnual {S}ymposium on {F}oundations of
              {C}omputer {S}cience---{FOCS} 2010},
     PAGES = {347--354},
 PUBLISHER = {IEEE Computer Soc., Los Alamitos, CA},
      YEAR = {2010},
   MRCLASS = {68W27 (06A06)},
  MRNUMBER = {3025208},
}

@article {bosek2015subexponential,
    AUTHOR = {Bosek, Bart{\l}omiej and Krawczyk, Tomasz},
     TITLE = {A subexponential upper bound for the on-line chain
              partitioning problem},
   JOURNAL = {Combinatorica},
  FJOURNAL = {Combinatorica. An International Journal on Combinatorics and
              the Theory of Computing},
    VOLUME = {35},
      YEAR = {2015},
    NUMBER = {1},
     PAGES = {1--38},
      ISSN = {0209-9683},
   MRCLASS = {68W27 (05C57 06A07)},
  MRNUMBER = {3341138},
       DOI = {10.1007/s00493-014-2908-7},
       URL = {https://doi.org/10.1007/s00493-014-2908-7},
}

\end{document}